\documentclass[12pt]{article}

\usepackage{fullpage}

\usepackage{ulem}
\usepackage{amssymb,amsmath,amsthm}

\usepackage{xcolor}
\usepackage{graphicx} 
\usepackage{multicol}
\graphicspath{{./}{figures/}} 
\usepackage{tikz}
\usetikzlibrary{matrix}
\usepackage{hyperref}

\newtheorem{theorem}{Theorem}[section]
\newtheorem{lemma}[theorem]{Lemma}

\newtheorem{proposition}[theorem]{Proposition}

\newtheorem{corollary}[theorem]{Corollary}
\newtheorem{remark}[theorem]{Remark}

\numberwithin{equation}{section}

\title{Policy iteration method for time-dependent Mean Field Games systems with non-separable Hamiltonians}

\author{Mathieu~Lauri{\`e}re\thanks{NYU Shanghai, China} \and Jiahao Song\thanks{China university of Geosciences (Wuhan), China} \and Qing Tang \footnotemark[2]}
\date{\today} 

\begin{document}
\maketitle

\begin{abstract}
We introduce two algorithms based on a policy iteration method to numerically solve time-dependent Mean Field Game systems of partial differential equations with non-separable Hamiltonians. We prove the convergence of such algorithms in sufficiently small time intervals with Banach fixed point method. Moreover, we prove that the convergence rates are linear. We illustrate our theoretical results by numerical examples, and we discuss the performance of the proposed algorithms.
\end{abstract}

\noindent
{\footnotesize \textbf{AMS-Subject Classification:} 49N80; 35Q89; 91A16; 65N12}.
\footnotetext[1]{Corresponding author: Qing Tang, tangqingthomas@gmail.com.}
\footnotetext[2]{The authors thank Professor Fabio Camilli (University of Rome ``La Sapienza") for helpful discussions.}
\section{Introduction}
Mean Field Games (MFG for short) theory has been introduced in \cite{hcm,ll} to characterize Nash equilibria for differential games involving a large (infinite) number of agents. The corresponding mathematical formulation
leads to the study of a system of partial differential equations (PDEs), composed by a Hamilton-Jacobi-Bellman (HJB for short) equation, characterizing the value function
and the optimal control for the agents; and a Fokker-Planck (FP for short) equation, governing the distribution of the   population when the agents behave in an optimal way. For a comprehensive introduction to the applications of MFG theory we refer to the monographs by Carmona and Delarue \cite{carmona2018}, Bensoussan, Frehse and Yam~\cite{ben} and the lecture notes~\cite{achdouCetraro}. In the case of a problem with finite horizon $T>0$ and periodic boundary conditions, the MFG system reads as 
\begin{equation}\label{MFG}
\begin{cases}
	-\partial_tu-{\epsilon \Delta} u+H(m,Du)=0 & \text{ in }Q\\
	\partial_tm -{\epsilon \Delta} m-\textrm{div}(mH_p(m,Du))=0 & \text{ in }Q\\
	m(x,0)=m_0(x),\, u(x,T)=u_T(x) & \text{ in }{\mathbb{T}^d}\ ,
\end{cases}
\end{equation}
where $Q:={\mathbb{T}^d}\times[0,T]$ and ${\mathbb{T}^d}$ stands for the flat torus $\mathbb{R}^d / \mathbb{Z}^d$, and $\epsilon>0$. \par
The numerical solution to the system~\eqref{MFG} is of paramount importance for applications of the MFG theory. Since the two equations in~\eqref{MFG} are strongly coupled in a forward-backward structure they can be solved neither independently of each other nor jointly with a simple time-marching method. This is an intensive area of research and many methods have been proposed with successful applications, see e.g.~\cite{achdou2020,lauriere2021} and the references therein. Convergence of finite difference schemes or semi-Lagrangian schemes has been studied in~\cite{ad,acd} and~\cite{carlinisilva1,carlinisilva2} respectively. But to the best of our knowledge, convergence of algorithms to solve the discrete problems has been proved only for a few methods. Convergence of fictitious play~\cite{ch,perrinFP} and online mirror descent~\cite{hadikhanlooNonatomic,perolatOMD} has been proved for monotone MFG with separable Hamiltonian. An augmented Lagrangian method and primal-dual methods have been studied respectively in~\cite{benamoucarlier,andreev} and in~\cite{bricenostatio,bricenodyn,nurbekyansaude}, and convergence holds when the MFG has a variational structure. The convergence of a monotone-flow methods for MFG satisfying a monotonocity condition has been considered in \cite{almulla2017,Gomes2021} using a contraction argument. However, none of these methods cover the case of non-separable Hamiltonian with a generic structure. To solve MFG with such Hamiltonians, in the absence of a more sophisticated method, a possible natural approach is to use fixed point iterations, i.e., to alternatively update the population distribution and the individual player's (optimal) value function. However, the computation of the value function boils down to the resolution of a HJB equation, which is in general costly. Policy iteration (in the context of MFG) can be viewed as a modification the fixed point procedure in which, at each iteration the HJB is solved for a fixed control, which is updated separately after the update of the value function. The policy iteration method for MFG can also be viewed as an extension of the usual policy iteration method for HJB equations: here, the update of the population is intertwined with the updates of the value function and the control.   \par
In \cite{ccg}, Cacace, Camilli and Goffi introduced the policy iteration method to study the numerical approximation of the solution to the mean field games system~\eqref{MFG}, when the Hamiltonians have a separable structure, i.e., 
\begin{equation}\label{H separable}
	H(x,m,Du)={\cal H}(x,Du)+F(x,m(x,t)).
\end{equation}\par
They introduced suitable discretizations to numerically solve both stationary and time-dependent MFGs and they proved the convergence of the policy iteration method for the continuous and the discrete problems. Moreover, the performance of the algorithm on examples in dimension one and two has been discussed. The rate of convergence of this method has been considered in \cite{ct}. \par
The policy iteration method, introduced by Bellman \cite{b} and Howard \cite{h} is a method to solve nonlinear HJB  equations arising in discrete and continuous optimal control problems. The general principle is to replace the HJB equation by a sequence of linearized equations each with a fixed policy. The policy is then updated at each step by solving an optimization problem given the current value function. Recently, a policy iteration algorithm has been used in \cite{crr} for a mean field games model in mathematical finance.\par
The convergence of policy iteration algorithm for HJB equations has been widely studied (see \cite{alla,fl,pu1,pb,santos}). However, the convergence analysis of policy iteration for MFGs has distinct features from that of HJB equations. The policy iteration scheme for HJBs is known to be improving the solution monotonically \cite{bmz,kss}. In general, this monotonicity property is lost in the policy iteration for MFGs, as already observed in \cite{ccg}. Heuristically, this is because the Nash equilibrium problem is a fixed point problem, even though each single agent solves an optimization problem. \par
In many applications, the separable Hamiltonian assumption~\eqref{H separable} is considered to be too restrictive. Typical examples without separable Hamiltonians are MFGs with congestion effects (see e.g. \cite{achdou2016,achdou2015}) or MFGs arising in macroeconomics (see e.g.~\cite{achdou2014partialmacroecon}). From the PDE point of view, the short time existence of solutions to general MFGs with non-separable Hamiltonians has been studied in \cite{cirant2020,ccp}, a series of papers by Ambrose et al. \cite{ambrose2018,ambrose2020,ambrose2021} and Gangbo et al. in \cite{gangbo2021}. The probabilistic approach to this problem has been considered in \cite{carmona2018}. The existence and uniqueness of solution the congestion type MFGs has been studied by Achdou and Porretta in~\cite{achdouporretta2018}, Gomes et al. \cite{gomes2015, ferreira2021existence} and Graber \cite{graber2015}. For the numerical solution of congestion type MFGs we refer to \cite{achdoulasrycrowd,achdou2020,lauriere2021}. \par

Based on the ideas from \cite{ccg},  we propose the following two policy iteration algorithms for the MFG system~\eqref{MFG}.\par
We first define the Lagrangian as the Legendre transform of $H$:
\begin{equation}\label{Legendre transform}
	L(m,q)=\sup_{p\in\mathbb{R}^d}\{p\cdot q-H(m,p)\}.
\end{equation}\par

The first policy iteration algorithm consists in iteratively updating the population distribution, the value function and the control. We introduce a bound $R$ on the control. In applications this may \textcolor{red}{be} understood as, for example, some financial constraints \cite{crr}. All of our convergence results hold when $R$ is large enough and the control is unconstrained.\par 
\textbf{Policy iteration algorithm 1 (PI1):} Given $R>0$ and given a bounded, measurable vector field   $q^{(0)}:{\mathbb{T}^d}\times [0,T]\to\mathbb{R}^d$  with $\vert q^{(0)}\vert \leq R$ and $\|\textrm{div} q^{(0)}\|_{L^r(Q)}\leq R$,  iterate:
\begin{itemize}
	\item[\textbf{(i)}] Solve
	\begin{equation}\label{alg_FP}
		\left\{
		\begin{array}{ll}
			\partial_t m^{(n)}-{\epsilon \Delta} m^{(n)}-\textrm{div} (m^{(n)} q^{(n)})=0,\quad &\text{ in }Q\\
			m^{(n)}(x,0)=m_0(x)&\text{ in }{\mathbb{T}^d}.
		\end{array}
		\right.
	\end{equation}
	%%%%	
	\item[\textbf{(ii)}] Solve
	\begin{equation}\label{alg_HJB}
		\left\{
		\begin{array}{ll}
			-\partial_t u^{(n)}- {\epsilon \Delta} u^{(n)}+q^{(n)} Du^{(n)}-{\cal{L}}(m^{(n)},Du^{(n-1)},q^{(n)})=0&\text{ in }Q\\
			u^{(n)}(x,T)=u_T(x)&\text{ in }{\mathbb{T}^d},
		\end{array}
		\right.
	\end{equation}
	where 
	for $(m,p,q) \in \mathbb{R}^+ \times \mathbb{R}^d \times \mathbb{R}^d$,
	\begin{equation}\label{eq:PI1_def_cL}
		{\cal{L}}(m,p,q)=p\cdot q - H (m,p).
	\end{equation}
	%%%%%%%%%%%	
	\item[\textbf{(iii)}] Update the policy
	\begin{equation*}\label{eq:PI1_update_policy}
		q^{(n+1)}(x,t)={\arg\max}_{\vert q\vert\leq R}\left\{q\cdot Du^{(n)}(x,t)-L(m^{(n)},q)\right\}\quad\text{ in }Q.
	\end{equation*}
	\end{itemize}\par
%%%%%%%%%%%%%%%%%%%%%%%%%%%%%%%%%%%%%%%%%%%%%%%%%%%%%%%%%%%%
Since a change of the population distribution might induce a change in the control, a variant of the above method consists in updating the control between each update of the population distribution and the value function. This leads to a second version of the policy iteration algorithm. \par
\textbf{Policy iteration algorithm 2 (PI2):} Given $R>0$ and given a bounded, measurable vector field   $q^{(0)}:{\mathbb{T}^d}\times [0,T]\to\mathbb{R}^d$ with $\vert q^{(0)}\vert \leq R$ and $\|\textrm{div} q^{(0)}\|_{L^r(Q)}\leq R$,  iterate:
\begin{itemize}
	\item[\textbf{(i)}] Solve
	\begin{equation}%\label{eq:alg_FP}
		\left\{
		\begin{array}{ll}
			\partial_t m^{(n)}-{\epsilon \Delta} m^{(n)}-\textrm{div} (m^{(n)} q^{(n)})=0,\quad &\text{ in }Q\\
			m^{(n)}(x,0)=m_0(x)&\text{ in }{\mathbb{T}^d}.
		\end{array}
		\right.
	\end{equation}
	%%%%	
	\item[\textbf{(ii)}]  Update the policy
	\begin{equation}\label{eq:PI2_update_policy1}
		\tilde q^{(n)}(x,t)={\arg\max}_{\vert q\vert\leq R}\left\{\tilde q \cdot D\tilde u^{(n-1)}(x,t)-L(m^{(n)},\tilde q)\right\}\quad\text{ in }Q.
	\end{equation}

	\item[\textbf{(iii)}] Solve
	\begin{equation}\label{eq:PI2_alg_HJ}
		\left\{
		\begin{array}{ll}
			-\partial_t \tilde u^{(n)}- {\epsilon \Delta} \tilde u^{(n)}+\tilde q^{(n)} D\tilde u^{(n)}-L(m^{(n)},\tilde q^{(n)})=0&\text{ in }Q,\\
			\tilde u^{(n)}(x,T)=u_T(x)&\text{ in }{\mathbb{T}^d}.
		\end{array}
		\right.
	\end{equation}
		%%%%%%%%%%%	
	\item[\textbf{(iv)}] Update the policy
	\begin{equation}\label{eq:PI2_update_policy}
		q^{(n+1)}(x,t)={\arg\max}_{\vert q\vert\leq R}\left\{q\cdot D\tilde u^{(n)}(x,t)-L(m^{(n)},q)\right\}\quad\text{ in }Q.
	\end{equation}
\end{itemize}\par
It is important to notice that while $L$ appearing in~\eqref{eq:PI1_update_policy},~\eqref{eq:PI2_update_policy1},~\eqref{eq:PI2_alg_HJ}, and~\eqref{eq:PI2_update_policy} is the Lagrangian, the term ${\cal{L}}$ defined in~\eqref{eq:PI1_def_cL} and appearing in~\eqref{alg_HJB} can only be regarded as a ``perturbed Lagrangian" since it can not be obtained by the Legendre transform~\eqref{Legendre transform}. The difference between the algorithms \textbf{(PI1)} and \textbf{(PI2)} is specific to the non-separable Hamiltonian structure. For a MFG system with~\eqref{H separable} the control depends on $m$ only implicitly via $u$ and hence $\tilde q^{(n)} = q^{(n)}$. However, in the non-separable case the dependence is explicit. Therefore it can be helpful to update the control again after each update of $m$.\par
If we use a separable Hamiltonian, then both algorithms \textbf{(PI1)} and \textbf{(PI2)} will be the same as the one proposed in \cite{ccg}, where the authors proved the convergence using a compactness argument under the Lasry Lions monotonicity condition. In this paper, we concentrate on the case of MFGs with non-separable Hamiltonians. The existence and uniqueness of solutions to such systems, in many cases, can only be obtained by restricting to a short time interval, or assuming smallness of data (e.g. \cite{ambrose2018,ambrose2020,ambrose2021}). This is particularly true when we consider Hamiltonians which can be degenerate at $m=0$. Assuming the time horizon $T$ is sufficiently small, we prove via the contraction fixed point method the convergence of both algorithms \textbf{(PI1)} and \textbf{(PI2)} to the solution of the MFG system~\eqref{MFG}. Furthermore,  we prove that the convergence takes place at a linear rate for both schemes without the additional assumptions on the Hamiltonians as in \cite{ct}. In our paper, the notion of solution for the Fokker-Planck equation is more regular than the one used for \cite{ccg}.\par        
As in \cite{ccg} for the separable Hamiltonian case, an important advantage of our method is that at each iteration we only need to solve two PDEs that are linear and decoupled. %the linear HJB and FP equations are completely decoupled. 
The advantage in terms of computational time is illustrated in the numerical examples by comparing with a fixed point algorithm combined with Newton-type method to solve the non-linear HJB equation. \par
The paper is organized as follows. In Section~\ref{sec:prelim}, we introduce some notations, assumptions and preliminary results. In Section~\ref{sec:pi} we prove the convergence of the policy iteration algorithms in a sufficiently small time interval (see Theorems~\ref{thm:policy_iteration} and~\ref{thm:policy_iteration2} for each policy iteration method). Restricting our attention to a small time horizon is justified by the fact that we want to avoid making restrictive assumptions on the structure of the MFG (see also Remark~\ref{rem:congestion-example}). In Section~\ref{sec:estimate}, we prove the linear rates of convergence for the two MFG policy iteration schemes under additional assumptions on the time interval (see Theorems~\ref{converge rate 1} and~\ref{converge rate 2}). In Section~\ref{sec:numerical} we provide numerical examples to illustrate our results.

%%%%%%%%%%%%%%%%%%%%%%%%%%%%%%%%%%%%%%%%%%%%%%%%%%%%%
%                                                   %
%%%%%%%%%%%%%%%%%%%%%%%%%%%%%%%%%%%%%%%%%%%%%%%%%%%%%

\section{Preliminaries}\label{sec:prelim}
We recall some basic facts on Legendre transform that are repeatedly used throughout the paper. We denote 
$$
f^*(p)=\sup_{\vert q\vert \leq R}\{p\cdot q-f(q)\}.
$$\par
For a strictly convex $f(q)$ with suitable regularity assumptions, the supremum for $f^*(p)$ is attained at $q^*$, where 
\begin{equation}\label{q}
q^*_i(x,t)=\partial_{p_i}f^*(p)(x,t)=
\begin{cases}
	(\partial_qf(\cdot))^{-1}p(x,t) & \text{ if } \vert(\partial_{q_i}f(\cdot))^{-1}p_i(x,t)\vert \leq R,\\
	R\text{sign}(p_i(x,t)) & \text{ if } \vert(\partial_{q_i}f(\cdot))^{-1}p_i(x,t)\vert > R.
	\end{cases}
\end{equation}
\par
For more details about the quadratic Hamiltonian with control constraints and its numerical approximation, we refer to sections 5 and 6 in \cite{all}.\par
We now introduce some useful anisotropic Sobolev spaces to handle time-dependent problems. First, given a Banach space $X$, $L^p(0,T;X)$ denotes the usual vector-valued Lebesgue space. For any $r\geq1$, we denote by $W^{2,1}_r(Q)$ the space of functions $f$ such that $\partial_t^{{\delta}}D^{\sigma}_x u\in L^r(Q)$ for all multi-indices $\sigma$ and ${\delta}$ such  that $\vert \sigma \vert+2{\delta}\leq  2$, endowed with the norm
\begin{equation*}
	\|u\|_{W^{2,1}_r(Q)}=\Big(\int_{Q}\sum_{\vert \sigma \vert+2{\delta}\leq2}\vert \partial_t^{{\delta}}D^{\sigma}_x u\vert ^rdxdt\Big)^{\frac1r}.
\end{equation*}
We recall that, by classical results in interpolation theory, the sharp  space of initial (or terminal) trace of $W^{2,1}_r(Q)$ is given by the fractional Sobolev class $W^{2-\frac{2}{r}}_r({\mathbb{T}^d})$.\par

 $\vert u\vert$ denotes the usual $L^\infty(Q)$ norm for $u(x,t)$ with $(x,t)\in Q$. $C^{1,0}(Q)$ with the norm $\vert u\vert ^{(1)}_Q$ will be the space of continuous functions on $Q$ with continuous derivatives in the $x-$variable, up to the parabolic boundary, since the spatial variable varies in the torus, up to $t=0$.\par
We recall the definition of parabolic H\"older spaces on the torus (we refer to \cite{LSU} for a more comprehensive discussion). For $0<\alpha<1$, we denote 
\begin{equation}\label{def Holder}
[u]_{C^{\alpha,\frac{\alpha}{2}}(Q)}:=\sup_{(x_1,t_1),(x_2,t_2)\in Q}\frac{\vert u(x_1,t_1)-u(x_2,t_2)\vert }{(d(x_1,x_2)^2+\vert t_1-t_2\vert )^{\frac{\alpha}{2}}} ,
\end{equation}
where $d(x,y)$ stands for the geodesic distance from $x$ to $y$ in ${\mathbb{T}^d}$. The parabolic H\"older space $C^{\alpha,\frac{\alpha}{2}}(Q)$ is the space of functions $u\in L^\infty(Q)$ for which $[u]_{C^{\alpha,\frac{\alpha}{2}}(Q)}<\infty$. It is endowed with the norm:
\begin{equation*}
\vert u\vert^{(\alpha)}_Q=\vert u\vert+[u]_{C^{\alpha,\frac{\alpha}{2}}(Q)}.
\end{equation*}\par
The space $C^{1+\alpha,\frac{1+\alpha}{2}}(Q)$ is endowed with the semi-norm
\begin{equation}\label{Holder alpha+1}
[u]_{C^{1+\alpha,\frac{1+\alpha}{2}}(Q)}:=\sum_{i=1}^d\vert \partial_{x_i}u\vert^{(\alpha)}_Q+\sup_{(x_1,t_1),(x_2,t_2)\in Q}\frac{\vert u(x_1,t_1)-u(x_2,t_2)\vert }{\vert t_1-t_2\vert ^{\frac{1+\alpha}{2}}} ,
\end{equation}
and the norm
\begin{equation}\label{Holder alpha+1 norm}
\vert u\vert^{(1+\alpha)}_Q=\vert u\vert+[u]_{C^{1+\alpha,\frac{1+\alpha}{2}}(Q)}.
\end{equation}\par
 Likewise, $\vert \cdot \vert^{(\alpha)}_{{\mathbb{T}^d}}$ and $\vert \cdot \vert ^{(1)}_{{\mathbb{T}^d}}$ are used to define the analogous norms on spaces of functions defined on ${\mathbb{T}^d}$. \par
For a vector field $b(x,t)$ in $\mathbb{R}^d$ with $1\leq i\leq d$ we denote
$$
\|b\|_{L^r(Q)}=\sup_i \|b_i\|_{L^r(Q)},\,\,\,\vert b\vert =\sup_{i}\vert b_i\vert.
$$
\begin{lemma}\label{Holder T} (Lemma 2.3 of \cite{cirant2020}). Let $\alpha \in (0,1)$. For any $f\in C^{1+\alpha,(1+\alpha)/2}(Q)$,
\begin{align}
	\vert f\vert ^{(1)}_Q\leq \vert f(\cdot,T)\vert ^{(1)}_{{\mathbb{T}^d}}+T^{\alpha/2}\vert f\vert ^{(1+\alpha)}_Q,\\
	\vert f\vert ^{(1)}_Q\leq \vert f(\cdot,0)\vert ^{(1)}_{{\mathbb{T}^d}}+T^{\alpha/2}\vert f\vert ^{(1+\alpha)}_Q.
\end{align}
\end{lemma}
\begin{lemma}\label{Sobolev T} (Lemma 2.4 of \cite{cirant2020}) Let $r>1$, $f\in W^{2,1}_{2r}(Q)$. Then
\begin{equation}
	\| f\| _{W^{2,1}_{r}(Q)}\leq T^{\frac{1}{2r}}\| f\| _{W^{2,1}_{2r}(Q)}.
\end{equation}
\end{lemma}

\begin{proposition}\label{Holder embedding} (Inequality (2.21) of \cite{gianni1995}, or Proposition 2.5 of \cite{cirant2020}). Let $f\in W^{2,1}_r(Q)$. Then,
\begin{equation}
\vert f\vert ^{(2-\frac{d+2}{r})}_Q\leq C(\| f\| _{W^{2,1}_r(Q)}+\| f(\cdot,0)\| _{W^{2-\frac{2}{r}}_r({\mathbb{T}^d})}),\,\,r>\frac{d+2}{2},\,r\neq d+2,
\end{equation}
where $C$ remains bounded for bounded values of $T$.
\end{proposition}
This embedding result is distinct from the classic result (Corollary p.342 of \cite{LSU}) in that the constant $C$ remains bounded when $T$ tends to $0$.\par
We can then easily obtain the following from Lemma~\ref{Holder T} and Proposition~\ref{Holder embedding}.
\begin{lemma}\label{Holder embedding2}
Let $r>d+2$, $\bar{f}\in W^{2,1}_{r}(Q)$. We assume either $\bar{f}(x,0)=0$ or $\bar{f}(x,T)=0$. Then
\begin{equation}
\vert \bar{f}\vert ^{(1)}_Q\leq CT^{\frac{1}{2}-\frac{d+2}{2r}}\| \bar{f}\| _{W^{2,1}_r(Q)},
\end{equation}
where $C$ remains bounded for bounded values of $T$.
\end{lemma}
We describe the assumptions on the data of the problem. 
\begin{itemize}
	\item[\textbf{(H1)}] $H$ is continuous with respect to $p, m$. $H, H_{p_i}, H_{p_ip_j}, H_{mp_i}$ are locally Lipschitz continuous functions with respect to $(p,m)\in \mathbb{R}^d \times \mathbb{R}^+$.	\item[\textbf{(H2)}] $H$ is strictly convex in the $p$-entry.
	\end{itemize}

\begin{remark}\label{rem:congestion-example} Typical examples we are going to consider are MFGs with congestion effects. For instance, with ${\gamma}>1$, $\beta$ and $c$ non-negative constants and $f(m)$ a locally Lipschitz function of $m \in \mathbb{R}^+$, we consider:
	$$
	H(m,p)=\frac{\vert p\vert ^{{\gamma}}}{(c+m)^{\beta}}+f(m).
	$$
	Note that when $\beta>2$, the Hamiltonian is super-quadratic and when $c=0$ the Hamiltonian is degenerate at $m=0$. Covering such situations is one of the reasons why we restrict our attention to short time horizons for the convergence results we prove in the sequel. 
%
%\begin{itemize}
%        \item Local coupling:
%	$$
%H(m,p)=\frac{\vert p\vert ^{{\gamma}}}{c+m^{\b}}+f(m),
%$$
%	 
%	 \item Nonlocal coupling:
%	$$
%H(m,p)=\frac{\vert p\vert ^{{\gamma}}}{c+\rho*m^{\b}},
%$$
% $\rho$ is a regularizing kernel. 
%\end{itemize}

\end{remark}
\begin{remark}
In this paper we only consider the case $H(m,p)$. The results can be naturally extended to include Hamiltonians of the form $H(x,m,p)$ with suitable additional assumptions.
\end{remark}
	
In the following we recall a classical result of the linear parabolic equation.  
\begin{equation}\label{linear parabolic}
		\begin{cases}
			-\partial_t u-{\epsilon \Delta} u+b(x,t)\cdot  Du+c(x,t)u=f(x,t)&\text{ in }Q,\\
			u(x,T)=u_T(x)&\text{ in }{\mathbb{T}^d}.
		\end{cases}
	\end{equation}
\begin{proposition}\label{linear estim}
	Let $r>d+2$ and suppose that $b\in L^\infty(Q;\mathbb{R}^d)$, $c\in L^\infty(Q)$, $f\in L^r(Q)$, and $u_T\in W^{2-\frac{2}{r}}_r({\mathbb{T}^d})$. Then the problem \eqref{linear parabolic} admits a unique solution $u\in W^{2,1}_r(Q)$  and it holds
	\begin{equation}\label{estim}
		\| u\| _{W^{2,1}_r(Q)}\leq C(\| f\| _{L^r(Q)}+\| u_T\| _{W^{2-\frac{2}{r}}_r({\mathbb{T}^d})}),
	\end{equation}
	where $C$ depends on the $L^\infty(Q)$ norms of the coefficients $b$ and $c$ as well as on $r$, $d$, $T$ and remains bounded for bounded values of $T$. \par
\end{proposition}
This result has been proved with other boundary conditions  (\cite{LSU}, Chapter 9, Theorem 9.1) and proved in the periodic setting in Appendix A of \cite{cirant2020}, where the proof is based on the local estimate of \cite{LSU}, eq. (10.12), p.355. The case when $c\in L^r(Q)$ has been considered in the Appendix of \cite{bhp}.
\begin{proposition}\label{linear estim2}
(Theorem 4, Appendix of \cite{bhp}) Let $r>d+2$. Suppose that $b\in L^r(Q;\mathbb{R}^d)$, $c\in L^r(Q)$, $f\in L^r(Q)$,  and $u_T\in W^{2-\frac{2}{r}}_r({\mathbb{T}^d})$. The problem \eqref{linear parabolic} admits a unique solution $u\in W^{2,1}_r(Q)$ such that
$$
\|u\|_{W^{2,1}_r(Q)}\leq C,
$$
 with $C$ depending only on the $L^r(Q)$ norms of $b$, $c$, $f$ and $\| u_T\| _{W^{2-\frac{2}{r}}_r({\mathbb{T}^d})} $.	
\end{proposition}
\begin{remark}Proposition \ref{linear estim2} only established the existence and uniqueness of solutions $u\in W^{2,1}_r(Q)$ but not the estimate \eqref{estim}. With the technique developed in \cite{cirant2020} one can show that \eqref{estim} holds under the assumptions of Proposition \ref{linear estim2} and $T$ sufficiently small.
\end{remark}

\section{Policy iterations methods for the MFG system}\label{sec:pi}
We will use the following assumption. 
\begin{itemize}
	\item[\textbf{(I1)}] 
	$u_T\in W^{2}_\infty({\mathbb{T}^d})$;
	for some $r>d+2$, $m_0\in W^{2}_r({\mathbb{T}^d})$, $m_0\geq \underline{m}> 0$ and $\int_{{\mathbb{T}^d}}m_0(x)dx=1$.
\end{itemize}
We define the space
\begin{equation}\label{X_M^T}
\begin{split}
	X_M^T=\{(u,m):u\in C^{1,0}(Q)\cap W^{2,1}_r(Q),m\in C^{1,0}(Q),\\
	 \| u\| _{W^{2,1}_r(Q)}+\vert u\vert ^{(1)}_Q+\vert m\vert ^{(1)}_Q\leq M\}.
\end{split}
\end{equation}

\subsection{Policy iteration \textbf{(PI1)}}
Let us start with the analysis of the policy iteration method \textbf{(PI1)}. 
Define the operator $\mathcal{T}$ on $X_M^T$ by: $\mathcal{T}(u,m)=(\hat u,\hat m)$ such that
\begin{equation}\label{fixed point system}
		\left\{
		\begin{array}{ll}
			\partial_t \hat m-{\epsilon \Delta} \hat m-H_p(m,Du)D\hat m-H_{pm}(m,Du)(Dm)\hat m\\
			\quad -\sum_{i,j}H_{p_jp_i}(m,Du)(\partial^2_{x_jx_i}u)\hat m=0\\
			-\partial_t \hat u- {\epsilon \Delta} \hat u+H_p(m,Du) D\hat u-{\cal{L}}(\hat m,Du,H_p(m,Du))=0,\\
			\hat u(x,T)=\hat u_T(x), \quad 
			\hat m(x,0)=\hat m_0(x),
		\end{array}
		\right.
	\end{equation}
where we recall that the perturbed Lagrangian ${\cal{L}}$ is defined in~\eqref{eq:PI1_def_cL}.\par
The following theorem and its proof are highly based on Theorem 1.1 of \cite{cirant2020}. However, since it is central to the theoretical study of the policy iteration algorithms, we provide the full details of the proof.
\begin{theorem}\label{fixed point 1} 
Suppose \textbf{(H1)}, \textbf{(H2)} and \textbf{(I1)} hold. Let $K$ be such that 
\begin{equation}\label{K condition}
	K\geq \max \{\frac{2}{\underline{m}},2\vert m_0\vert^{(1)}_{\mathbb{T}^d},2\vert u(\cdot,T)\vert ^{(1)}_{\mathbb{T}^d}\}.
\end{equation}
Let
\begin{equation}\label{M_1}
M_1 =2\big(\vert m_0\vert ^{(1)}_{{\mathbb{T}^d}}+\vert u_T\vert ^{(1)}_{{\mathbb{T}^d}}\big). 
\end{equation}
Then there exists $\bar{T}$ sufficiently small such that for all $T\in (0,\bar{T}]$, $\mathcal{T}$ is a contraction on the space 
$$
X_{M_1}^T\cap \{ (u,m) \,:\, \vert u\vert ^{(1)}_Q\leq K,\,1/K\leq m\leq K\}.
$$
\end{theorem}

\begin{proof}
\textbf{Step 1: Lipschitz regularization.}  Let $\varphi$ be a global Lipschitz function such that $\varphi(z)=z$ for all $z\in [1/K,K]$, $\varphi(z)\in [1/(2K),2K]$ for all $z\in \mathbb{R}$. 
Let $\psi(z):\mathbb{R}^d\rightarrow \mathbb{R}^d$ be a globally Lipschitz function such that $\psi(z)=z$ for all $\vert p\vert \leq K$ and $\vert \psi(z)\vert \leq 2K$ for all $z\in \mathbb{R}^d$.\par
We then consider the regularized fixed point operator $\mathcal{T}_K$ defined on $X_M^T$ by: $\mathcal{T}_K(u,m)=(\hat u,\hat m)$ such that
\begin{equation}\label{regularized fixed point system}
		\left\{
		\begin{array}{ll}
			\partial_t \hat m-{\epsilon \Delta} \hat m-H_p(\varphi(m),\psi(Du))D\hat m-(H_{pm}(\varphi(m),\psi(Du))(Dm)\hat m\\
			\quad -\sum_{i,j}H_{p_jp_i}(\varphi(m),\psi(Du))(\partial^2_{x_jx_i}u)\hat m=0\\
			-\partial_t \hat u- {\epsilon \Delta} \hat u+H_p(\varphi(m),\psi(Du)) D\hat u-{\cal{L}}(\varphi(\hat m),\psi(Du),H_p(\varphi(m),\psi(Du)))
			=0,\\
			\hat u(x,T)=\hat u_T(x), \quad 
			\hat m(x,0)=\hat m_0(x),
		\end{array}
		\right.
	\end{equation}
where ${\cal{L}}$ is defined in~\eqref{eq:PI1_def_cL}.
%\begin{align*}
%&{\cal{L}}(\varphi(\hat m),\psi(Du),H_p(\varphi(m),\psi(Du)))\\
%=&H_p(\varphi(m),\psi(Du))\psi(Du)-H(\varphi(\hat m),\psi(Du)).
%\end{align*}

From \textbf{(H1)} and~\eqref{K condition} we have $H_p(\varphi(m),\psi(Du))$, $H_{pp}(\varphi(m),\psi(Du))$ and $H_{pm}(\varphi(m),\psi(Du))$ are globally Lipschitz with respect to the pair $(m,Du)$. In fact, there exist two positive constants $C_H$ and $C_H'$ depending only on the data of the problem and $K$, such that for all $(u_1,m_1), (u_2,m_2) \in X_M^T$,
\begin{equation}\label{H_p Lip bound}
\begin{split}
	&\sup_{(x,t)\in Q} \Big\{\vert H(\varphi(m_1),\psi(Du_1))-H(\varphi(m_2),\psi(Du_2))\vert \\
	&+\vert H_p(\varphi(m_1),\psi(Du_1))-H_p(\varphi(m_2),\psi(Du_2))\vert \\
	&+\vert  H_{pp}(\varphi(m_1),\psi(Du_1))-H_{pp}(\varphi(m_2),\psi(Du_2))\vert \\
	&+\vert  H_{pm}(\varphi(m_1),\psi(Du_1))-H_{pm}(\varphi(m_2),\psi(Du_2))\vert \Big\}\\
	\leq {}&C_H(\vert m_1-m_2\vert ^{(1)}_Q+\vert u_1-u_2\vert ^{(1)}_Q),
\end{split}
\end{equation}
and for all $(u,m) \in X_M^T$,
\begin{equation}\label{H bound}
\begin{split}
	&\sup_{(x,t)\in Q} \Big\{\vert H(\varphi(m),\psi(Du))\vert +\vert H_p(\varphi(m),\psi(Du))\vert  \\
	& +\vert H_{pp}(\varphi(m),\psi(Du))\vert +\vert H_{pm}(\varphi(m),\psi(Du))\vert \Big\}
	\leq C'_H.
\end{split}
\end{equation}

From~\eqref{H_p Lip bound} we obtain that there exists $C_{{\cal{L}}}$ depending only on $K$ such that
\begin{equation}\label{L Lip bound}
\begin{split}
	&\sup_{(x,t)\in Q} \Big\{\vert {\cal{L}}(\varphi(\hat m_1),\psi(Du_1),H_p(\varphi(m_1),\psi(Du_1)))\\
	& -{\cal{L}}(\varphi(\hat m_2),\psi(Du_2),H_p(\varphi(m_2),\psi(Du_2)))\vert \Big\}\\
	\leq
	&\sup_{(x,t)\in Q} \Big\{\vert H_p(\varphi(m_1),\psi(Du_1))-H_p(\varphi(m_2),\psi(Du_2))\vert \cdot \vert 	\psi(Du_1)\vert \\
	& +\vert \psi(Du_1)-\psi(Du_2)\vert \cdot \vert H_p(\varphi(m_2),\psi(Du_2))\vert \\
	&+\vert H(\varphi(\hat m_1),\psi(Du_1))-H(\varphi(\hat m_2),\psi(Du_2))\vert \Big\}\\
	\leq {}&C_{{\cal{L}}}\big((\vert m_1-m_2\vert ^{(1)}_Q+\vert \hat m_1-\hat m_2\vert ^{(1)}_Q+\vert u_1-u_2\vert ^{(1)}_Q\big).
\end{split}
\end{equation}
We denote by $C'_{{\cal{L}}}$ a constant such that: 
\begin{equation}
\sup_{(x,t)\in Q}\vert {\cal{L}}(\varphi(\hat m),\psi(Du),H_p(\varphi(m),\psi(Du)))\vert \leq C'_{{\cal{L}}}.
\end{equation}

\textbf{Step 2: $\mathcal{T}_K$ maps $X_{M_1}^T$ into itself.}\\
Suppose that $(u,m)\in X_{M_1}^T$, then from \eqref{H bound} and $\| u\| _{W^{2,1}_r(Q)}\leq M_1$, we have
\begin{equation}\label{H_pp}
\begin{split}
\| H_{p_jp_i}(\varphi(m),\psi(Du))\partial_{x_jx_i}u\| _{L^r(Q)}\leq {}&\| H_{p_jp_i}(\varphi(m),\psi(Du))\| _{L^\infty(Q)}\cdot \| \partial^2_{x_jx_i}u\| _{L^r(Q)}\\
\leq {}&C'_H M_1.
\end{split}
\end{equation}
Likewise we have 
\begin{equation}\label{H_p}
\| H_p(\varphi(m),\psi(Du))\| _{L^r(Q)}\leq T^{1/r}C'_H,
\end{equation}
\begin{equation}\label{H_pm}
\| H_{pm}(\varphi(m),\psi(Du))Dm\| _{L^r(Q)}\leq T^{1/r}C'_HM_1,
\end{equation}
where $C'_H$ and $M_1$ are defined in \eqref{H bound} and \eqref{M_1}.\par
From~\eqref{H_pp}, \eqref{H_p} and \eqref{H_pm}, we can obtain using Proposition~\ref{linear estim2} that there exists a unique solution $\hat m$ to the first equation in \eqref{regularized fixed point system} such that
$$
\| \hat m\| _{W^{2,1}_r(Q)}\leq C_1,
$$
where $C_1$ depends only on $K$, $M_1$, $C'_H$ and $\| m_0\| _{W^{2-\frac{2}{r}}_r({\mathbb{T}^d})}$. \par
From Proposition~\ref{Holder embedding} we have
\begin{equation}\label{C_2}
\vert \hat m\vert ^{(2-\frac{d+2}{r})}_Q \leq C_2(\| \hat m\| _{W^{2,1}_r(Q)}+\| m_0\| _{W^{2-\frac{2}{r}}_r({\mathbb{T}^d})})\leq C'_2.
\end{equation}
 Together with Lemma~\ref{Holder T} this yields
\begin{equation}
\vert \hat m\vert ^{(1)}_Q\leq \vert m_0\vert ^{(1)}_{{\mathbb{T}^d}}+T^{\frac{1}{2}-\frac{d+2}{2r}}C'_2.
\end{equation}
Here both $C_2$ and $C'_2$ remain bounded for bounded $T$.\par
\indent Next, we consider the linearized HJB equation in the system \eqref{regularized fixed point system}. Again from Proposition~\ref{linear estim} we have
\begin{align*}
	\| \hat u\| _{W^{2,1}_{2r}(Q)}
	&\leq C_3(\| {\cal{L}}(\varphi(\hat m),\psi(Du),H_p(\varphi(m),\psi(Du)))\| _{L^{2r}(Q)}+\| u_T\| _{W^{2-\frac{1}{r}}_{2r}({\mathbb{T}^d})})\\
	&\leq T^{\frac{1}{2r}}C_3C_{{\cal{L}}}+C_3\| u_T\| _{W^{2-\frac{1}{r}}_{2r}({\mathbb{T}^d})}\\
	&\leq C'_3.
\end{align*}
Using Lemma~\ref{Sobolev T} and Lemma~\ref{Holder T} we obtain 
\begin{equation}
\label{eq:bound-hat-u}
	\| \hat u\| _{W^{2,1}_{r}(Q)}\leq T^{\frac{1}{2r}}C'_3,
	\quad \vert \hat u\vert ^{(1)}_Q\leq \vert u_T\vert ^{(1)}_{{\mathbb{T}^d}}+T^{\frac{1}{2}-\frac{d+2}{2r}}C'_3.
\end{equation}
 Here again, both $C_3$ and $C'_3$ remain bounded for bounded $T$.
Recall \eqref{M_1} and $r>d+2$. Then there exists $T$ so small that 
\begin{align*}
&\vert \hat m\vert ^{(1)}_Q+\vert \hat u\vert ^{(1)}_Q+\| \hat u\| _{W^{2,1}_{r}(Q)}\\
\leq {}&\vert m_0\vert ^{(1)}_{{\mathbb{T}^d}}+T^{\frac{1}{2}-\frac{d+2}{2r}}C'_2+\vert u_T\vert ^{(1)}_{{\mathbb{T}^d}}+T^{\frac{1}{2}-\frac{d+2}{2r}}C'_3+T^{\frac{1}{2r}}C'_3\\
<{}&M_1.
\end{align*}
Therefore, we have $\mathcal{T}: X_{M_1}^T\rightarrow X_{M_1}^T$. \par

\textbf{Step 3: $\mathcal{T}_K: X_{M_1}^T\rightarrow X_{M_1}^T$ is a contraction operator.}\\
Let $(\hat u_1,\hat m_1) := \mathcal{T}(u_1,m_1)$ and $(\hat u_2,\hat m_2):=\mathcal{T}(u_2,m_2)$. We aim at showing
\begin{equation}
\begin{split}
& \| \hat u_1-\hat u_2\| _{W^{2,1}_r(Q)}+\vert \hat u_1-\hat u_2\vert ^{(1)}_Q+\vert \hat m_1-\hat m_2\vert ^{(1)}_Q\\
 \leq {}&\Gamma \big(\| u_1-u_2\| _{W^{2,1}_r(Q)}+\vert u_1-u_2\vert ^{(1)}_Q+\vert m_1-m_2\vert ^{(1)}_Q\big),
\end{split}
\end{equation}
for some $0<\Gamma<1$. \par
Denote $\bar{U}=\hat u_1-\hat u_2$ and $\bar{M}=\hat m_1-\hat m_2$. From system~\eqref{regularized fixed point system} we have
\begin{equation}\label{bar M}
\begin{split}
	&\partial_t \bar{M}-{\epsilon \Delta} \bar{M}-H_p(\varphi(m_1),\psi(Du_1))D\bar{M}\\
	&-\Big(H_{pm}(\varphi(m_1),\psi(Du_1))Dm_1 +\sum_{i,j}H_{p_jp_i}(\varphi(m_1),\psi(Du_1))\partial^2_{x_jx_i}u_1\Big)\bar{M}\\
	&-\Big(H_p(\varphi(m_1),\psi(Du_1))-H_p(\varphi(m_2),\psi(Du_2))\Big)D\hat m_2\\
	&-\Big(H_{pm}(\varphi(m_1),\psi(Du_1))Dm_1+\sum_{i,j}H_{p_jp_i}(\varphi(m_1),\psi(Du_1))\partial^2_{x_jx_i}u_1\Big)\hat m_2\\
	&+\Big(H_{pm}(\varphi(m_2),\psi(Du_2))Dm_2+\sum_{i,j}H_{p_jp_i}(\varphi(m_2),\psi(Du_2))\partial^2_{x_jx_i}u_2\Big)\hat m_2\\
	&=0.
\end{split}
\end{equation}
From~\eqref{H_p Lip bound} and the fact that $\hat m_2$ remains in $X_{M_1}^T$ we have
\begin{equation}
\begin{split}
	&\sup_{(x,t) \in Q}\vert \big(H_p(\varphi(m_1),\psi(Du_1))-H_p(\varphi(m_2),\psi(Du_2))\big)D\hat m_2\vert \\
	\leq {}&C_HM_1(\vert u_1-u_2\vert ^{(1)}_Q+\vert m_1-m_2\vert ^{(1)}_Q).
\end{split}
\end{equation}
Moreover,
\begin{equation}
\begin{split}
	&\sup_{(x,t) \in Q} \Big\{\vert H_{pm}(\varphi(m_1),\psi(Du_1))Dm_1-H_{pm}(\varphi(m_2),\psi(Du_2))Dm_2\vert \Big\} \\
		\leq &\sup_{(x,t) \in Q} \Big\{ \vert H_{pm}(\varphi(m_1),\psi(Du_1))-H_{pm}(\varphi(m_2),\psi(Du_2))\vert \cdot \vert Dm_1\vert \\
	&\quad+\vert H_{pm}(\varphi(m_2),\psi(Du_2))\vert \cdot \vert Dm_1-Dm_2\vert \Big\} \\
	\leq {}&(M_1C_H+C'_H)\big(\vert u_1-u_2\vert ^{(1)}_Q+\vert m_1^{(1)}-m_2\vert ^{(1)}_Q\big).
\end{split}
\end{equation}
Since $\vert \hat m_2\vert \leq M_1$, $\| u_1\| _{W^{2,1}_r(Q)}\leq M_1$ and $\| u_2\| _{W^{2,1}_r(Q)}\leq M_1$, we have
\begin{align*}
	&\| \big(H_{p_jp_i}(\varphi(m_1),\psi(Du_1))\partial^2_{x_jx_i}u_1-H_{p_jp_i}(\varphi(m_2),\psi(Du_2))\partial^2_{x_jx_i}u_2\big)\hat m_2\| _{L^r(Q)}\\
	\leq {}&M_1\| H_{p_jp_i}(\varphi(m_1),\psi(Du_1))-H_{p_jp_i}(\varphi(m_2),\psi(Du_2))\| _{L^\infty(Q)}\cdot \|\partial^2_{x_jx_i}u_1\| _{L^r(Q)}\\
	&+M_1\| H_{p_jp_i}(\varphi(m_1),\psi(Du_1))\| _{L^\infty(Q)}\cdot \| \partial^2_{x_jx_i}u_1-\partial^2_{x_jx_i}u_2\| _{L^r(Q)}\\
	\leq {}&C_HM_1^2\big( \vert m_1-m_2\vert ^{(1)}_Q+\vert u_1-u_2\vert ^{(1)}_Q\big)+M_1C'_H\| u_1-u_2\| _{W^{2,1}_r(Q)}.
\end{align*}
Therefore from Proposition \ref{linear estim2} we obtain that there exists a unique solution $\bar{M}$ to \eqref{bar M} such that  
$$
\| \bar{M}\| _{W^{2,1}_r(Q)}\leq C_4,
$$
where $C_4$ depends only on $K$, $M_1$, $C'_H$, $C_H$ and $\| m_0\| _{W^{2-\frac{2}{r}}_r({\mathbb{T}^d})}$. \par
Moreover, from Proposition \ref{linear estim} one can obtain that 
\begin{equation}
\begin{split}
	\| \bar{M}\| _{W^{2,1}_r(Q)}
	&\leq C'_4\big( \vert m_1-m_2\vert ^{(1)}_Q+\vert u_1-u_2\vert ^{(1)}_Q+\| u_1-u_2\| _{W^{2,1}_r(Q)}+\vert \bar{M}\vert \big)\\
	&\leq C'_4\big( \vert m_1-m_2\vert ^{(1)}_Q+\vert u_1-u_2\vert ^{(1)}_Q+\| u_1-u_2\| _{W^{2,1}_r(Q)}+\vert \bar{M}\vert ^{(1)}_Q \big),	
\end{split}
\end{equation}
where $C'_4$ remains bounded for bounded $T$. From Lemma~\ref{Holder embedding2} we have
\begin{equation*}
\vert \bar{M}\vert ^{(1)}_Q\leq T^{\frac{1}{2}-\frac{d+2}{2r}}C'_4\big(\| u_1-u_2\| _{W^{2,1}_r(Q)}+ \vert m_1-m_2\vert ^{(1)}_Q+\vert u_1-u_2\vert ^{(1)}_Q+\vert \bar{M}\vert ^{(1)}_Q \big).
\end{equation*}
Since $C'_4$ remains bounded for bounded $T$, it is then clear that for sufficiently small $T$, such that $T^{\frac{1}{2}-\frac{d+2}{2r}}C'_4<1$, we can obtain
\begin{equation}\label{M1Q bound}
\vert \bar{M}\vert ^{(1)}_Q\leq T^{\frac{1}{2}-\frac{d+2}{2r}}C''_4\big(\| u_1-u_2\| _{W^{2,1}_r(Q)}+ \vert m_1-m_2\vert ^{(1)}_Q+\vert u_1-u_2\vert ^{(1)}_Q\big),
\end{equation}
where $C''_4$ remains bounded for bounded $T$. \par
\indent We next turn to the linearized HJB equation. From~ \eqref{regularized fixed point system}  we have
\begin{equation}
\begin{split}
	0=&-\partial_t \bar{U}-{\epsilon \Delta} \bar{U}+H_p(\varphi(m_1),\psi(Du_1))D\bar{U}\\
	&+(H_p(\varphi(m_1),\psi(Du_1))-H_p(\varphi(m_2),\psi(Du_2)))Du_2\\
	&-{\cal{L}}(\varphi(\hat m_1),\psi(Du_1),H_p(\varphi(m_1),\psi(Du_1)))\\
	&+{\cal{L}}(\varphi(\hat m_2),\psi(Du_2),H_p(\varphi(m_2),\psi(Du_2))).
\end{split}
\end{equation}
From~\eqref{L Lip bound} and~\eqref{M1Q bound} we have
\begin{align*}
	&\| {\cal{L}}(\varphi(\hat m_1),\psi(Du_1),H_p(\varphi(m_1),\psi(Du_1)))\\
	&-{\cal{L}}(\varphi(\hat m_2),\psi(Du_2),H_p(\varphi(m_2),\psi(Du_2)))\| _{L^r(Q)}\\
	\leq {}&T^{1/r}C_{{\cal{L}}}(\vert \hat m_1-\hat m_2\vert ^{(1)}_Q+\vert m_1-m_2\vert ^{(1)}_Q+\vert u_1-u_2\vert ^{(1)}_Q)\\
	\leq {}&C_{{\cal{L}}}(T^{1/r}+T^{\frac{1}{2}-\frac{d}{2r}}C''_4)( \| u_1-u_2\| _{W^{2,1}_r(Q)}+\vert m_1-m_2\vert ^{(1)}_Q+\vert u_1-u_2\vert ^{(1)}_Q).
\end{align*}
From~\eqref{H_p Lip bound},
\begin{align*}
	&\| (H_p(\varphi(m_1),\psi(Du_1))-H_p(\varphi(m_2),\psi(Du_2)))Du_2\| _{L^r(Q)}\\
	\leq {}&T^{1/r}M_1C_H( \vert m_1-m_2\vert ^{(1)}_Q+\vert u_1-u_2\vert ^{(1)}_Q).
\end{align*}

Again, from Proposition~\ref{linear estim} we have 
\begin{align*}
\| \bar{U}\| _{W^{2,1}_r(Q)}
\leq {}&C_5\big(\| (H_p(\varphi(m_1),\psi(Du_1))-H_p(\varphi(m_2),\psi(Du_2)))Du_2\| _{L^r(Q)}\\
&+\vert \vert {\cal{L}}(\varphi(\hat m_1),\psi(Du_1),H_p(\varphi(m_1),\psi(Du_1)))\\
&-{\cal{L}}(\varphi(\hat m_2),\psi(Du_2),H_p(\varphi(m_2),\psi(Du_2)))\vert \vert _{L^r(Q)}\big)\\
\leq {} & C_5 \big(C_{{\cal{L}}}(T^{1/r}+T^{\frac{1}{2}-\frac{d}{2r}}C_4)+T^{1/r}M_1C_H\big) \big(\| u_1-u_2\| _{W^{2,1}_r(Q)}\\
&+ \vert m_1-m_2\vert ^{(1)}_Q+\vert u_1-u_2\vert ^{(1)}_Q\big).
\end{align*}

By using Lemma~\ref{Holder T} and Proposition~\ref{Holder embedding} we have
\begin{align*}
\vert \bar{U}\vert ^{(1)}_Q
\leq {}&T^{\frac{1}{2}-\frac{d+2}{2r}}C_6 \big(C_{{\cal{L}}}(T^{1/r}+T^{\frac{1}{2}-\frac{d}{2r}}C_4)+T^{1/r}M_1C_H\big) \big(\| u_1-u_2\| _{W^{2,1}_r(Q)}\\
&+ \vert m_1-m_2\vert ^{(1)}_Q+\vert u_1-u_2\vert ^{(1)}_Q\big).
\end{align*}
Therefore we obtain
\begin{equation}
\begin{split}
 &\| \hat u_1-\hat u_2\| _{W^{2,1}_r(Q)}+\vert \hat u_1-\hat u_2\vert ^{(1)}_Q+\vert \hat m_1-\hat m_2\vert ^{(1)}_Q\\
 \leq {}&(T^{1/r}+T^{\frac{1}{2}-\frac{d}{2r}}+T^{\frac{1}{2}-\frac{d+2}{2r}})C_7\big(\| u_1-u_2\| _{W^{2,1}_r(Q)}\\
 &+\vert u_1-u_2\vert ^{(1)}_Q+\vert m_1-m_2\vert ^{(1)}_Q\big),
 \end{split}
\end{equation}
where $C_7$ here remains bounded for bounded $T$. By making $T$ small enough we can obtain that $\mathcal{T}_K$ is a strict contraction in $X^T_{M_1}$. \par
Hence, from Banach fixed point theorem, if $T$ is small enough, $\mathcal{T}_K: X_{M_1}^T\rightarrow X_{M_1}^T$ admits a unique fixed point. We denote it by $(u^*,m^*)$.

\textbf{Step 4: Sobolev regularity.} Note that $(u^*,m^*)$ satisfies 
\begin{equation}
\begin{split}
\partial_t m^*-{\epsilon \Delta} m^*-\textrm{div} (H_{p}(\varphi(m^*),\psi(Du^*)))m^*
-H_p(\varphi(m^*),\psi(Du^*))Dm^*=0,
\end{split}
\end{equation}
and $\partial_{x_i} (H_p(\varphi(m^*),\psi(Du^*)))\in L^r(Q)$ and $H_p(\varphi(m^*),\psi(Du^*))\in L^r(Q)$. Hence we have $m^*\in W^{2,1}_r(Q)$ from Proposition~\ref{linear estim2}.\par
Moreover, from system \eqref{regularized fixed point system} we have $(\hat u,\hat m)$ bounded in $W^{2,1}_r(Q)$ with the bound independent of $(u,m)$, $\mathcal{T}_K(u,m)=(\hat u,\hat m)$.

\textbf{Step 5: back to the initial problem.} We conclude by showing that $\mathcal{T}$ maps the space
$$
X_{M_1}^T\cap \{\vert u\vert ^{(1)}_Q\leq K,\,1/K\leq m\leq K\},
$$
into itself, if $T$ is sufficiently small. It is then a contraction in this space. \par
From \eqref{Holder alpha+1} and \eqref{Holder alpha+1 norm}  we have, taking $t_1,t_2\in [0,T]$,
\begin{align*}
&\sup_{(x,t_1)\neq(x,t_2)\in Q}\frac{\vert \hat m(x,t_1)-\hat m(x,t_2)\vert }{T^{1-\frac{d+2}{2r}}}\\
\leq {}&\sup_{(x,t_1)\neq(x,t_2)\in Q}\frac{\vert \hat m(x,t_1)-\hat m(x,t_2)\vert }{\vert t_1-t_2\vert ^{1-\frac{d+2}{2r}}}\\
\leq {}&\vert \hat m\vert _Q^{(2-\frac{d+2}{r})}.
\end{align*}
Suppose $\hat m(x,t)$ attains its minimum at $(\hat{x},\hat{t})$, then since
$$
\frac{\vert \hat m(\hat{x},\hat{t})-\hat m(\hat{x},0)\vert }{T^{1-\frac{d+2}{2r}}}\leq \sup_{(x,t_1)\neq(x,t_2)\in Q}\frac{\vert \hat m(x,t_1)-\hat m(x,t_2)\vert }{T^{1-\frac{d+2}{2r}}},
$$
we have
$$
\hat m(\hat{x},0)-\hat m(\hat{x},\hat{t})\leq \vert \hat m(\hat{x},\hat{t})-\hat m(\hat{x},0)\vert \leq T^{1-\frac{d+2}{2r}}\vert \hat m\vert _Q^{(2-\frac{d+2}{r})}.
$$
From~\eqref{C_2} we have then 
\begin{align*}
\min_{(x,t)\in Q} \hat m(x,t)&\geq m(\hat{x},0)-\vert \hat m\vert _Q^{(2-\frac{d+2}{r})}T^{1-\frac{d+2}{2r}}\\
&\geq \underline{m}-C'_2T^{1-\frac{d+2}{2r}}.
\end{align*}
Likewise, we can get 
\begin{equation*}
\max_{(x,t)\in Q} \hat m\leq \vert \hat m\vert^{(1)}_{\mathbb{T}^d}+C'_2T^{1-\frac{d+2}{2r}}.
\end{equation*}
Moreover, as in~\eqref{eq:bound-hat-u}, 
\begin{equation}
\vert \hat u\vert ^{(1)}_Q\leq \vert u_T\vert ^{(1)}_{{\mathbb{T}^d}}+T^{\frac{1}{2}-\frac{d+2}{2r}}C'_3.
\end{equation}
We recall that both $C'_2$ and $C'_3$ remain bounded for bounded $T$, therefore we can choose $T$ so small that 
\begin{equation}
	C'_2T^{2-\frac{d+2}{r}}<\min \{ 1/K,K/2 \} ,\,\, T^{\frac{1}{2}-\frac{d+2}{2r}}C'_3<K/2.
\end{equation}

 This and~\eqref{K condition} yield $\vert \hat u\vert ^{(1)}_Q\leq K$ and $1/K\leq \hat m\leq K$ for all $(x,t)\in Q$.
\end{proof}

From Theorem \ref{fixed point 1}  we obtain the following short time existence and uniqueness result for the MFG system~\eqref{MFG} with non-separable Hamiltonian.
\begin{theorem}\label{existence1}
	Let \textbf{(H1)}, \textbf{(H2)} and \textbf{(I1)} be  in force. There exists a sufficiently small $\bar{T}$ such that for all $T\in (0,\bar{T}]$ the system~\eqref{MFG} admits a unique solution $(u^*,m^*) \in W^{2,1}_r(Q) \times W^{2,1}_r(Q)$.
\end{theorem}
%{\color{blue}
% This result is a special case of Theorem~4.1 from \cite{cirant2020} which allows $u_T$ to be a regularizing function of $m(\cdot,T)$, but excludes the case where $u_T$ depends locally on $m(\cdot,T)$ (see Section 3.1 of \cite{cirant2020}). In the present work, we focus on the local coupling case and hence we do not let $u_T$ depend on $m(\cdot,T)$.
% }
This result is a special case of Theorem~4.1 from \cite{cirant2020}. In our case, the terminal cost $u_T$ does not depend on $m(\cdot,T)$. In Theorem~4.1 of \cite{cirant2020} it may be a regularizing function of $m(\cdot,T)$, but excludes the case where $u_T$ depends locally on $m(\cdot,T)$ (see Section 3.1 of \cite{cirant2020}). Here we only consider terminal cost $u_T$ which does not depend on $m(\cdot,T)$, as the numerical approximation for nonlocal coupling with rigorous convergence analysis is beyond the scope of this paper.

The next result extends to the case with non-separable Hamiltonians the results \cite[Theorems 2.3 and 2.5]{ccg}.

\begin{theorem}\label{thm:policy_iteration}
		Let \textbf{(H1)},  \textbf{(H2)} and \textbf{(I1)} be  in force. Let $\bar{T}$ be as in Theorem~\ref{fixed point 1}. Then, there exists a $\hat{T}_1\leq \bar{T}$ such that for all $T\in (0,\hat{T}_1]$ and $R$ sufficiently large, the sequence $(u^{(n)},m^{(n)})$, generated by the policy iteration algorithm \textbf{(PI1)}, converges  to the solution  $(u^*,m^*)\in W^{2,1}_r(Q)\times  W^{2,1}_r(Q)$   of~\eqref{MFG}. 	
	\end{theorem}
\begin{proof} We start with an initial guess $q^{(0)}:{\mathbb{T}^d}\times [0,T]\to\mathbb{R}^d$ with $\vert q^{(0)}\vert<R$ and $\|\textrm{div} q^{(0)}\|_{L^r(Q)}<R$. We perform the same regularization as in \textbf{Step 1} of Theorem \ref{fixed point 1} with $\varphi$ and $\psi$, $K$ is defined by \eqref{K condition}. From \eqref{alg_FP}, \eqref{alg_HJB}, Proposition \ref{linear estim} and Proposition \ref{linear estim2}, using similar arguments as in \textbf{Step 5} of Theorem \ref{fixed point 1}, there exists a sufficiently small $T_1$ such that, for $T\in (0,T_1]$ we have $m^{(0)}\in [1/K,K]$ and $\vert Du^{(0)}\vert \leq K$. Then we start with the regularized iteration system, for $n\geq 0$:
\begin{equation}\label{regularized PI}
		\left\{
		\begin{array}{ll}
			\partial_t m^{(n+1)}-{\epsilon \Delta} m^{(n+1)}-H_p(\varphi(m^{(n)}),\psi(Du^{(n)}))Dm^{(n+1)}\\
			\quad -(H_{pm}(\varphi(m^{(n)}),\psi(Du^{(n)}))(Dm^{(n)})m^{(n+1)}\\
			\quad -\sum_{i,j}H_{p_jp_i}(\varphi(m^{(n)}),\psi(Du^{(n)}))\partial_{x_jx_i}u^{(n)}m^{(n+1)}=0\\
			-\partial_t u^{(n+1)}- {\epsilon \Delta} u^{(n+1)}+H_p(\varphi(m^{(n)}),\psi(Du^{(n)})) Du^{(n+1)}\\
			\quad -{\cal{L}}(\varphi(m^{(n+1)}),\psi(Du^{(n)}),H_p(\varphi(m^{(n)}),\psi(Du^{(n)})))=0,\\
			u^{(n+1)}(x,T)=u_T(x), \quad 
			m^{(n+1)}(x,0)=m_0(x). 
		\end{array}
		\right.
	\end{equation}
%	where 
%	\begin{align*}
%&{\cal{L}}(\varphi(m^{(n+1)}),\psi(Du^{(n)}),H_p(\varphi(m^{(n)}),\psi(Du^{(n)})))\\
%=&H_p(\varphi(m^{(n)}),\psi(Du^{(n)}))\psi(Du^{(n)})-H(\varphi(m^{(n+1)}),\psi(Du^{(n)})).
%\end{align*}
From the proof of Theorem~\ref{fixed point 1}, for $T\in (0,\bar{T}]$, we have that $(u^{(n)},m^{(n)})$ converges to the solution $(u^*,m^*)\in W^{2,1}_r(Q)\times  W^{2,1}_r(Q)$ of~\eqref{MFG}. We aim to show the iteration system~\eqref{regularized PI} is the same as \textbf{(PI1)}. \par
 We can argue inductively. For each $n$, assuming $m^{(n)}\in [1/K,K]$ and $\vert Du^{(n)}\vert \leq K$, we can follow the argument in \textbf{Step 5} of Theorem~\ref{fixed point 1} to obtain: for all $T\in (0,\bar{T}]$ we have $m^{(n+1)}\in [1/K,K]$ and $\vert Du^{(n+1)}\vert \leq K$. \par
By \textbf{(H1)} and the above remark, there exists a bound on 
$$
\vert H_p(m^{(n)},Du^{(n)}) \vert+\vert H_{pm}(m^{(n)},Du^{(n)}) \vert+\vert H_{pp}(m^{(n)},Du^{(n)}) \vert
$$ 
which depends only on $K$. From Proposition \ref{linear estim} and Proposition \ref{linear estim2}, we can also obtain a bound on $\vert m^{(n)}\vert^{(1)}_Q+\| u^{(n)}\|_{W^{2,1}_r(Q)}$ which depends only on data of the problem, $K$ and $T$, remains bounded for bounded $T$. Therefore we can obtain a bound for $\vert q^{(n)}\vert+\| \textrm{div} q^{(n)}\|_{L^r(Q)}$ independent of $n$. Then the system~\eqref{regularized PI} is exactly the algorithm \textbf{(PI1)}. \par
Finally we can conclude by choosing $\hat{T}_1=\min \{T_1, \bar{T} \}$.
\end{proof}

\begin{remark} In Theorem \ref{thm:policy_iteration}, we needed to introduce $\hat{T}_1\leq \bar{T}$ for considerations related to the initial guess. If the initial guess is sufficiently well chosen, then we may have $\hat{T}_1= \bar{T}$. In practice, we found that $q^{(0)}=0$ is usually a good initial guess. This can be partially explained as follows. In this case, \eqref{alg_FP} becomes a simple heat equation with initial condition $m(x,0)\in [2/K,K/2]$. By the maximum principle of heat equation we have directly that $m^{(0)}(x,t)\in [1/K,K]$, for all $(x,t)\in Q$ and $T\in (0,\bar{T}]$.
\end{remark}

\subsection{Policy iteration \textbf{(PI2)}}
We now turn our attention to the algorithm \textbf{(PI2)}. Recall that $X_M^T$ is defined in~\eqref{X_M^T}. 
We define the operator $\mathcal{T}_2$ on this set by: $\mathcal{T}_2(u,m)=(\hat u,\hat m)$ such that
\begin{equation}\label{fixed point system2}
		\left\{
		\begin{array}{ll}
			\partial_t \hat m-{\epsilon \Delta} \hat m-H_p(m,Du)D\hat m-H_{pm}(m,Du)(Dm)\hat m\\
			\quad -\sum_{i,j}H_{p_jp_i}(m,Du)(\partial_{x_jx_i}u)\hat m=0\\
			-\partial_t \hat u- {\epsilon \Delta} \hat u+H_p(\hat m,Du) D\hat u-L(\hat m,H_p(\hat m,Du))
			=0,\\
			\hat u(x,T)=u_T(x), \quad 
			\hat m(x,0)=m_0(x),
		\end{array}
		\right.
	\end{equation}
where
$
	L(\hat m,H_p(\hat m,Du))=H_p(\hat m,Du) Du-H(\hat m,Du).
$

\begin{theorem}\label{fixed point 2} 
Let \textbf{(H1)}, \textbf{(H2)} and \textbf{(I1)} be  in force. Then there exists $M_2$ sufficiently large and $\bar{T}_2$ sufficiently small such that for all $T\in (0,\bar{T}_2]$, $\mathcal{T}_2$ is a contractive operator in the space $X_{M_2}^T$.
\end{theorem}
We omit the proof since it is quite similar to the proof of Theorem~\ref{fixed point 1}.

\begin{theorem}\label{thm:policy_iteration2}
		Let \textbf{(H1)},  \textbf{(H2)} and \textbf{(I1)} be  in force and $\bar{T}_2$ be defined as in Theorem~\ref{fixed point 2}. Then, there exists a $\hat{T}_2\leq \bar{T}_2$, such that for all $T\in (0,\hat{T}_2]$ and $R$ sufficiently large, the sequence $(u^{(n)},m^{(n)})$, generated by the policy iteration algorithm \textbf{(PI2)}, converges  to the solution $(u^*,m^*)\in W^{2,1}_r(Q)\times  W^{2,1}_r(Q)$   of~\eqref{MFG}. 	
	\end{theorem}
We sketch the proof of Theorem~\ref{thm:policy_iteration2} to stress the main differences with the proof of Theorem~\ref{thm:policy_iteration}.
\begin{proof}
We start with an initial guess $q^{(0)}:{\mathbb{T}^d}\times [0,T]\to\mathbb{R}^d$ with $\vert q^{(0)}\vert\leq R$ and $\|\textrm{div} q^{(0)}\|_{L^r(Q)}\leq R$. We perform the same regularization as in \textbf{Step 1} of Theorem \ref{fixed point 1} with $\varphi$ and $\psi$, $K$ is defined by \eqref{K condition}. Using the same argument as in Theorem \ref{thm:policy_iteration} there exists a sufficiently small $T_2$ such that, for $T\in (0,T_2]$ we have $m^{(0)}\in [1/K,K]$ and $\vert Du^{(0)}\vert \leq K$. Then we start with the regularized iteration system, for $n\geq 0$:
\begin{equation}\label{regularized-PI2}
		\left\{
		\begin{array}{ll}
			\partial_t m^{(n+1)}-{\epsilon \Delta} m^{(n+1)}-H_p(\varphi(m^{(n)}),\psi(Du^{(n)}))Dm^{(n+1)}\\
			\quad-(H_{pm}(\varphi(m^{(n)}),\psi(Du^{(n)}))(Dm^{(n)})m^{(n+1)}\\
			\quad -\sum_{i,j}H_{p_jp_i}(\varphi(m^{(n)}),\psi(Du^{(n)}))\partial_{x_jx_i}u^{(n)}m^{(n+1)}=0,\\
			-\partial_t u^{(n+1)}- {\epsilon \Delta} u^{(n+1)}+H_p(\varphi(m^{(n+1)}),\psi(Du^{(n)})) Du^{(n+1)}\\
			\quad -L\big(\varphi(m^{(n+1)}),H_p(\varphi(m^{(n+1)}),\psi(Du^{(n)}))\big)=0,\\
			u^{(n+1)}(x,T)=u_T(x),\\
			m^{(n+1)}(x,0)=m_0(x).
		\end{array}
		\right.
	\end{equation}\par
 We can argue by induction: if $m^{(n)}\in [1/K,K]$, $\vert Du^{(n)}\vert \leq K$ and $T\in (0,\bar{T}_2]$, then we have $m^{(n+1)}\in [1/K,K]$ and $\vert Du^{(n+1)}\vert \leq K$. \par
There exist bounds depending only on $K$ for 
\begin{align*}
&\vert H_p(m^{(n)},Du^{(n)}) \vert+\vert H_{pm}(m^{(n)},Du^{(n)}) \vert+\vert H_{pp}(m^{(n)},Du^{(n)}) \vert\\
\text{and} \,\,\,&\vert H_p(m^{(n+1)},Du^{(n)}) \vert+\vert H_{pm}(m^{(n+1)},Du^{(n)}) \vert+\vert H_{pp}(m^{(n+1)},Du^{(n)}) \vert.
\end{align*}
We can then follow the same arguments as Theorem \ref{thm:policy_iteration} and obtain bounds on 
$$
\vert q^{(n)}\vert+\|{\rm{div} }q^{(n)}\|_{L^r(Q)}\,\,\, \text{and} \,\,\, \vert \tilde q^{(n)}\vert+\| {\rm{div} }\tilde q^{(n)}\|_{L^r(Q)},
$$
independent of $n$. Then the system~\eqref{regularized-PI2} is exactly the algorithm \textbf{(PI2)}. We can conclude by choosing $\hat{T}_2=\min \{T_2, \bar{T}_2 \}$.
\end{proof}

%%%%%%%%%%%%%%%%%%%%%%%%%%%%%%%%%%%%%%%%
%      Rate of convergence             %
%%%%%%%%%%%%%%%%%%%%%%%%%%%%%%%%%%%%%%%%
\section{A rate of convergence for the policy iteration method}\label{sec:estimate}

\begin{theorem}\label{converge rate 1}
	Let \textbf{(H1)}, \textbf{(H2)} and \textbf{(I1)} be  in force. Let $\hat{T}_1$ and $R$ be as in Theorem~\ref{thm:policy_iteration}. Then, there exists a constant $C$, which depends only on the data of problem and remains bounded for all $T\in (0,\hat{T}_1]$, such that, if $(u^{(n)}, m^{(n)})$ is the sequence generated by the policy iteration method \textbf{(PI1)}, we have
\begin{equation}\label{mLs}
	\| m^{(n+1)} - m^*\| _{W^{2,1}_r(Q)} \le C \big(\| q^{(n+1)}-q^*\| _{L^r (Q)}+\| {\rm{div}}(q^{(n+1)}-q^*)\| _{L^r (Q)}\big),
\end{equation}
and
	\begin{equation} \label{uW21r}
		\begin{split}
	 \| u^{(n+1)}-u^*\| _{W^{2,1}_r(Q)}
\le & CT^{\frac{1}{2}-\frac{d}{2r}}\big(\| u^{(n)}-u^*\| _{W^{2,1}_r(Q)}+\| m^{(n+1)}-m^*\| _{W^{2,1}_r(Q)}\big).
 	 \end{split}
	\end{equation}
\end{theorem}
%%%%%%%%%%%%%%%%
\begin{proof}
Along the proof, the constant $C$ can change from line to line, but it is always independent of $n$ and remains bounded for all $T\in (0,\hat{T}_1]$. We start with the proof of~\eqref{mLs} for the FP equation. For all $n$ we have $u^{(n)}, m^{(n)}\in W^{2,1}_r(Q)$. As in Theorem~\ref{fixed point 1} we take $(u,m) \in X_{M_1}^T$, so that
$$
\| u\| _{W^{2,1}_r(Q)}+\vert u\vert ^{(1)}_Q+\vert m\vert ^{(1)}_Q\leq M_1.
$$
From
$$
q^{(n)}=H_p(m^{(n-1)},Du^{(n-1)}),
$$
we have
\begin{equation}\label{pdq_j}
\partial_{x_i}q_j^{(n)}=H_{mp_j}(m^{(n-1)},Du^{(n-1)})\partial_{x_i}m^{(n-1)}+H_{p_ip_j}(m^{(n-1)},Du^{(n-1)})\partial^2_{x_ix_j}u^{(n-1)}.
\end{equation}
Hence for all $n$, ${\rm{div} }q^{(n)}\in L^r(Q)$.

%%%%%%%%%%%%  Stima FP %%%%%%%%%%%%%%%%%%%%
Set $M^{(n+1)}=m^{(n+1)}-m^*$. Then $M^{(n+1)}$ satisfies the equation
\begin{equation}\label{eq1}
	\partial_t M^{(n+1)}- {\epsilon \Delta} M^{(n+1)}-{\rm{div} } (q^{(n+1)} M^{(n+1)})={\rm{div} }((q^{(n+1)}-q^*) m^*),
\end{equation}
with $M^{(n+1)}(\cdot,0)=0$. This can be reformulated as
\begin{align*}
&\partial_t M^{(n+1)}- {\epsilon \Delta} M^{(n+1)}-q^{(n+1)} DM^{(n+1)}-\textrm{div} (q^{(n+1)}) M^{(n+1)}\\
=\,\,&{\rm{div} } (q^{(n+1)}-q^*) m^*+ (q^{(n+1)}-q^*) Dm^*.
\end{align*}
Since the $L^r(Q)$ norms of both $q^{(n+1)}$ and $\textrm{div} q^{(n+1)}$ are bounded independently of $n$, from Proposition~\ref{linear estim2}, assuming $T$ small, we have 
\begin{align*}
\| M^{(n+1)}\| _{W^{2,1}_r(Q)}\leq {}&C\| {\rm{div} } (q^{(n+1)}-q^*) m^*+ (q^{(n+1)}-q^*) Dm^*\| _{L^r (Q)}\\
\leq {}&C\vert m^*\vert^{(1)}_Q(\| q^{(n+1)}-q^*\| _{L^r(Q)}+\| {\rm{div} }(q^{(n+1)}-q^*)\| _{L^r (Q)})\\
\leq {}&C(\| q^{(n+1)}-q^*\| _{L^r(Q)}+\| {\rm{div} }(q^{(n+1)}-q^*)\| _{L^r (Q)}).
\end{align*}

\indent We now prove the estimate~\eqref{uW21r} for the HJB equation. The function  $U^{(n+1)}=u^{(n+1)}-u^*$ satisfies the equation
\begin{equation*}%\label{eq:conv1}
-\partial_t U^{(n+1)}- {\epsilon \Delta} U^{(n+1)}+q^{(n+1)} DU^{(n+1)}={\cal F}(x,t)
\end{equation*}
with $U^{(n+1)}(\cdot,T)=0$, where 
\begin{align*}
	{\cal F}(x,t)= {}&H(m^*,Du^*)-(q^{(n+1)}Du^*-{\cal{L}}(m^{(n+1)},Du^{(n)},q^{(n+1)}))\\
	= {}&H(m^*,Du^*)-H(m^{(n+1)},Du^{(n)})+q^{(n+1)}(Du^{(n)}-Du^*).
\end{align*}
Hence, recalling that $ q^{(n+1)} =H_p(m^{(n)},Du^{(n)})$ is bounded, again from Proposition~\ref{linear estim} we have
\begin{equation}\label{eq:conv}
	\| U^{(n+1)}\| _{W^{2,1}_r(Q)}\le  C\| {\cal F} \| _{L^r(Q)}.
\end{equation}
As we have assumed that conditions for Theorem~\ref{fixed point 1} are satisfied, then we have $m^*, m^{(n+1)}\in [\frac{1}{K},K]$, $\vert Du^*\vert \leq K$, $\vert Du^{(n)}\vert\leq K$ and $\vert q^{(n+1)}\vert \leq R$. Therefore from \eqref{H_p Lip bound} we have
$$
H(m^*,Du^*)-H(m^{(n+1)},Du^{(n)})\leq C_H\big(\| Du^{(n)}-Du^*\| _{L^{\infty}(Q)}+\| m^{(n+1)}-m^*\| _{L^\infty(Q)}\big),
$$
and then
\begin{align*}
	&\| {\cal F}\| _{L^r(Q)}\\
	\leq {}&T^{\frac{1}{r}}\| {\cal F}\| _{L^\infty(Q)}\\
	\leq {}&T^{\frac{1}{r}}\big(\| Du^{(n)}-Du^*\| _{L^{\infty}(Q)}+\| m^{(n+1)}-m^*\| _{L^\infty(Q)}\big)\\
	\leq {}&T^{\frac{1}{r}}\big(\vert u^{(n)}-u^*\vert ^{(1)}_Q+\vert m^{(n+1)}-m^*\vert ^{(1)}_Q\big)\\
	\leq {}&T^{\frac{1}{r}}\cdot CT^{\frac{1}{2}-\frac{d+2}{2r}}\big(\| u^{(n)}-u^*\| _{W^{2,1}_r(Q)}+\| m^{(n+1)}-m^*\| _{W^{2,1}_r(Q)}\big)\\
	&(\text{from Lemma \ref{Holder embedding2}})\\
	 \leq {}& CT^{\frac{1}{2}-\frac{d}{2r}}\big(\| u^{(n)}-u^*\| _{W^{2,1}_r(Q)}+\| m^{(n+1)}-m^*\| _{W^{2,1}_r(Q)}\big).
\end{align*}

Then we can get~\eqref{uW21r} from~\eqref{eq:conv}.
\end{proof}

\begin{corollary}\label{cor:rate}
Denote $M^{(n)}=m^{(n)}-m^*$, $U^{(n)}=u^{(n)}-u^*$ for all $n\geq 1$. Under the same assumptions as in Theorem~\ref{converge rate 1}, the following estimate holds for $n {\gamma}e 1$: 
\begin{equation}\label{estimate_rate}
\begin{split}
&\| U^{(n+1)}\| _{W^{2,1}_r(Q)}+\| M^{(n+1)}\| _{W^{2,1}_r(Q)}\\
\le {}&CT^{\frac{1}{2}-\frac{d}{2r}} \big(\| U^{(n)}\| _{W^{2,1}_r(Q)}+\| M^{(n)}\| _{W^{2,1}_r(Q)}+\| U^{(n-1)}\| _{W^{2,1}_r(Q)}\big).
\end{split}
	\end{equation}
Moreover, there exist $\ell>1$, $n_0$ sufficiently large and $\breve{T}\in (0,\hat{T}_1]$ sufficiently small such that, for all $T\in (0,\breve{T}]$, we have a linear rate of convergence, i.e., 
\begin{equation*}
	\| U^{(n)}\| _{W^{2,1}_r(Q)}+\| M^{(n)}\| _{W^{2,1}_r(Q)}\leq \left(\frac{1}{\ell}\right)^{n-n_0}\big(\| U^{(n_0)}\| _{W^{2,1}_r(Q)}+\| M^{(n_0)}\| _{W^{2,1}_r(Q)}\big).
\end{equation*}
\end{corollary}

\begin{proof} Along the proof, the constant $C$ can change from line to line, but it is always independent of $n$ and remains bounded for all $T\in (0,\hat{T}_1]$. \par
	First note that, from 
	$$
	q^{(n+1)}-q^*=H_p(m^{(n)},Du^{(n)})-H_p(m^*,Du^*),
	$$
	and \textbf{(H1)}  we have
\begin{equation*}
\begin{split}
	\| q^{(n+1)}-q^*\| _{L^r(Q)}\leq {}&CT^{\frac{1}{r}}\big(\vert u^{(n)}-u^*\vert ^{(1)}_Q+\vert m^{(n)}-m^*\vert ^{(1)}_Q\big)\\
	\leq {}&CT^{\frac{1}{2}-\frac{d}{2r}}\big(\| U^{(n)}\| _{W^{2,1}_r(Q)}+\| M^{(n)}\| _{W^{2,1}_r(Q)}\big).
\end{split}
\end{equation*}
From~\eqref{pdq_j} and \textbf{(H1)} we have
\begin{align*}
&\| \partial_{x_i}q_j^{(n+1)}-\partial_{x_i}q_j^*\| _{L^r(Q)}\\
={}&\| H_{mp_j}(m^{(n)},Du^{(n)})\partial_{x_i}m^{(n)}-H_{mp_j}(m^*,Du^*)\partial_{x_i}m^*\| _{L^r(Q)}\\
&+\| H_{p_ip_j}(m^{(n)},Du^{(n)})\partial^2_{x_ix_j}u^{(n)}-H_{p_ip_j}(m^*,Du^*)\partial^2_{x_ix_j}u^*\| _{L^r(Q)}\\
\leq {}&CT^{\frac{1}{r}}\vert m^{(n)}-m^*\vert ^{(1)}_Q+M_1\| H_{mp_j}(m^{(n)},Du^{(n)})-H_{mp_j}(m^*,Du^*)\| _{L^r(Q)}\\
&+M_1\| H_{p_ip_j}(m^{(n)},Du^{(n)})-H_{p_ip_j}(m^*,Du^*)\| _{L^r(Q)}\\
&+C\| \partial^2_{x_ix_j}u^{(n)}-\partial^2_{x_ix_j}u^*\| _{L^r(Q)},
\end{align*}
and 
\begin{equation*}
\begin{split}
	&\| \textrm{div} q^{(n+1)}-\textrm{div} q^*\| _{L^r(Q)} \\
	\leq {}&C(\| u^{(n)}-u^*\| _{W^{2,1}_r(Q)}+T^{\frac{1}{r}}\vert m^{(n)}-m^*\vert ^{(1)}_Q\big)\\
	\leq {}&C(\| u^{(n)}-u^*\| _{W^{2,1}_r(Q)}+T^{\frac{1}{2}-\frac{d}{2r}}\| m^{(n)}-m^*\| _{W^{2,1}_r(Q)}\big).
\end{split}
\end{equation*}
	By~\eqref{uW21r} and the fact that $r>d+2$ we have
\begin{align*}
&\| m^{(n+1)} - m^*\| _{W^{2,1}_r(Q)} \\
\le {}&C\big(\| q^{(n+1)}-q^{*}\| _{L^r(Q)}+\| {\rm{div} }(q^{(n+1)}-q^*)\| _{L^r(Q)}\big)\\
\le {}&C\big(\| u^{(n)}-u^*\| _{W^{2,1}_r(Q)}+T^{\frac{1}{2}-\frac{d}{2r}}\| m^{(n)} - m^*\| _{W^{2,1}_r(Q)}\big)\\
\le {}&CT^{\frac{1}{2}-\frac{d}{2r}} \big(\| u^{(n-1)}-u^*\| _{W^{2,1}_r(Q)}+\| m^{(n)} - m^*\| _{W^{2,1}_r(Q)}\big).
\end{align*}
From the previous estimates, \eqref{estimate_rate} follows.\par
Since $(U^{(n)},M^{(n)})$ converges to $(0,0)$, there exist $\ell>1$ and $n_0$ sufficiently large such that 
\begin{equation*}
	\| U^{(n_0+1)}\| _{W^{2,1}_r(Q)}+\| M^{(n_0+1)}\| _{W^{2,1}_r(Q)}\leq \frac{1}{\ell}\big(\| U^{(n_0)}\| _{W^{2,1}_r(Q)}+\| M^{(n_0)}\| _{W^{2,1}_r(Q)}\big).
\end{equation*}
By choosing $T$ so small that $CT^{\frac{1}{2}-\frac{d}{2r}}\le \frac{1}{\ell+\ell^2}$ and $T\leq \hat{T}_1$, by induction, we have for all $n\geq n_0$
\begin{equation*}
	\| U^{(n+2)}\| _{W^{2,1}_r(Q)}+\| M^{(n+2)}\| _{W^{2,1}_r(Q)}\leq (\frac{1}{\ell})^2\big(\| U^{(n)}\| _{W^{2,1}_r(Q)}+\| M^{(n)}\| _{W^{2,1}_r(Q)}\big).
\end{equation*}
\end{proof}

\begin{theorem}\label{converge rate 2}
	Let \textbf{(H1)}, \textbf{(H2)} and \textbf{(I1)} be  in force and $\hat{T}_2$ and $R$ be defined as in Theorem~\ref{thm:policy_iteration2}. Then, there exists a constant $C$, which depends only on the data of problem and remains bounded for all $T\in (0,\hat{T}_2]$, such that, if $(u^{(n)}, m^{(n)})$ is the sequence generated by the policy iteration method \textbf{(PI2)}, we have
\begin{equation}\label{mLs2}
\| m^{(n+1)} - m^*\| _{W^{2,1}_r(Q)} \le C \big(\| q^{(n+1)}-q^*\| _{L^r (Q)}+\| {\rm{div} }(q^{(n+1)}-q^*)\| _{L^r (Q)}\big),
\end{equation}
and
	\begin{equation} \label{uW21r2}
		\begin{split}
	 \| \tilde u^{(n+1)}-u^*\| _{W^{2,1}_r(Q)}
\le & CT^{\frac{1}{2}-\frac{d}{2r}}\big(\| \tilde u^{(n)}-u^*\| _{W^{2,1}_r(Q)}+\| m^{(n+1)}-m^*\| _{W^{2,1}_r(Q)}\big).
 	 \end{split}
	\end{equation}\end{theorem}

\begin{proof}
Along the proof, the constant $C$ can change from line to line, but it is always independent of $n$ and remains bounded for all $T\in (0,\hat{T}_2]$. \par
Again, set $M^{(n+1)}=m^{(n+1)}-m^*$. We repeat the exact same reasoning as Theorem~\ref{converge rate 1} to obtain
\begin{equation*}
\| M^{(n+1)}\| _{W^{2,1}_r(Q)}\leq C(\| q^{(n+1)}-q^*\| _{L^r(Q)}+\| {\rm{div} }(q^{(n+1)}-q^*)\| _{L^r (Q)}).
\end{equation*}
The function  $\tilde u^{(n+1)}=\tilde u^{(n+1)}-u^*$ satisfies the equation
\begin{equation*}%\label{eq:conv1}
-\partial_t \tilde u^{(n+1)}- {\epsilon \Delta} \tilde u^{(n+1)}+\tilde q^{(n+1)} D\tilde u^{n+1}=\tilde{{\cal F}}(x,t)
\end{equation*}
with $\tilde u^{(n+1)}(x,T)=0$, where 
\begin{align*}
\tilde{{\cal F}}(x,t)=&H(m^*,Du^*)-(\tilde q^{(n+1)}Du^*-L(m^{(n+1)},\tilde q^{(n+1)}))\\
=&H(m^*,Du^*)-H(m^{(n+1)},D\tilde u^{(n)})+\tilde q^{(n+1)}(D\tilde u^{(n)}-Du^*).
\end{align*}
Hence, recalling that $ \tilde q^{(n+1)} =H_p(m^{(n+1)},D\tilde u^{(n)})$ is bounded, again from Proposition~\ref{linear estim} we have
\begin{equation}\label{eq:conv1}
\| \tilde u^{(n+1)}\| _{W^{2,1}_r(Q)}\le  C\| \tilde{{\cal F}} \| _{L^r(Q)},
\end{equation}
and (following the same arguments as in Theorem~\ref{converge rate 1})
\begin{equation*}
\| \tilde{{\cal F}}(x,t)\| _{L^r(Q)} \le CT^{\frac{1}{2}-\frac{d}{2r}}\big(\| \tilde u^{(n)}-u^*\| _{W^{2,1}_r(Q)}+\| m^{(n+1)}-m^*\| _{W^{2,1}_r(Q)}\big).
\end{equation*}
Then we can get~\eqref{uW21r2} from~\eqref{eq:conv1}.
\end{proof}

\begin{corollary}
Denote $M^{(n)}=m^{(n)}-m^*$, $\tilde u^{(n)}=\tilde u^{(n)}-u^*$ for all $n\geq 1$. Under the same assumptions of Theorem~\ref{converge rate 2}, the following estimate holds
\begin{equation}
\begin{split}
&\| \tilde u^{(n+1)}\| _{W^{2,1}_r(Q)}+\| M^{(n+1)}\| _{W^{2,1}_r(Q)}\\
\le {}&CT^{\frac{1}{2}-\frac{d}{2r}} \big(\| \tilde u^{(n)}\| _{W^{2,1}_r(Q)}+\| M^{(n)}\| _{W^{2,1}_r(Q)}+\| \tilde u^{(n-1)}\| _{W^{2,1}_r(Q)}\big).
\end{split}
	\end{equation}
Moreover, there exists $\ell>1$, $n_0$ sufficiently large and $\Breve{\Breve{T}} \in (0,\hat{T}_2]$ sufficiently small such that, for all $T\in (0,\Breve{\Breve{T}} ]$, we have a linear rate of convergence
\begin{equation}\label{estimate_rate2}
\| \tilde u^{(n)}\| _{W^{2,1}_r(Q)}+\| M^{(n)}\| _{W^{2,1}_r(Q)}\leq (\frac{1}{\ell})^{n-n_0}\big(\| \tilde u^{(n_0)}\| _{W^{2,1}_r(Q)}+\| M^{(n_0)}\| _{W^{2,1}_r(Q)}\big).
\end{equation}
\end{corollary} 

\begin{proof} Along the proof, the constant $C$ can change from line to line, but it is always independent of $n$ and remains bounded for all $T\in (0,\hat{T}_2]$. From the results of Theorem~\ref{converge rate 2} we have
	\begin{equation}
\begin{split}
&\| \textrm{div} q^{(n+1)}-\textrm{div} q^*\| _{L^r(Q)} \\
\leq\,\, &C(\| \tilde u^{(n)}-u^*\| _{W^{2,1}_r(Q)}+T^{\frac{1}{2}-\frac{d}{2r}}\| m^{(n)}-m^*\| _{W^{2,1}_r(Q)}\big),
\end{split}
\end{equation}
	and
\begin{align*}
	&\| m^{(n+1)} - m^*\| _{W^{2,1}_r(Q)} \\
	\le\,\, &C\big(\| q^{(n+1)}-q^{*}\| _{L^r(Q)}+\| \textrm{div}(q^{(n+1)}-q^*)\| _{L^r(Q)}\big)\\
	\le\,\, &C\big(\| \tilde u^{(n)}-u^*\| _{W^{2,1}_r(Q)}+T^{\frac{1}{2}-\frac{d}{2r}}\| m^{(n)} - m^*\| _{W^{2,1}_r(Q)}\big)\\
	\le\,\, &CT^{\frac{1}{2}-\frac{d}{2r}} \big(\| \tilde u^{(n-1)}-u^*\| _{W^{2,1}_r(Q)}+\| m^{(n)} - m^*\| _{W^{2,1}_r(Q)}\big).
\end{align*}
From the previous estimates, it follows~\eqref{estimate_rate2}.\par
Since $(\tilde u^{(n)},M^{(n)})$ converges to $(0,0)$, there exists a $\ell>1$, $n_0$ sufficiently large such that 
\begin{equation*}
	\| \tilde u^{(n_0+1)}\| _{W^{2,1}_r(Q)}+\| M^{(n_0+1)}\| _{W^{2,1}_r(Q)}
	\leq \frac{1}{\ell}\big(\| \tilde u^{(n_0)}\| _{W^{2,1}_r(Q)}+\| M^{(n_0)}\| _{W^{2,1}_r(Q)}\big),
\end{equation*}
by choosing $T$ so small that $CT^{\frac{1}{2}-\frac{d}{2r}}\le \frac{1}{\ell+\ell^2}$ and $T\leq \hat{T}_2$, by induction, we have for all $n\geq n_0$
\begin{equation*}
	\| \tilde u^{(n+2)}\| _{W^{2,1}_r(Q)}+\| M^{(n+2)}\| _{W^{2,1}_r(Q)}\leq (\frac{1}{\ell})^2\big(\| \tilde u^{(n)}\| _{W^{2,1}_r(Q)}+\| M^{(n)}\| _{W^{2,1}_r(Q)}\big).
\end{equation*}
\end{proof}

\section{Numerical simulations}\label{sec:numerical}
In this section, we illustrate with numerical examples the two policy iteration methods analyzed in the previous sections. To this end we rely on a finite difference scheme introduced and analyzed in~\cite{ad,acd} for MFG PDE systems. 
We consider the following examples.

\vskip 6pt
\noindent\textbf{Example 1:} We first consider the following one dimensional example in which the agents are encouraged to move towards one of two possible targets. They are penalized at the terminal time based on the distance to the nearest target. This is reminiscent of the min-LQG MFG of~\cite{salhab2017collectivechoice}, except that we do not consider mean field interactions that encourage the agents to follow the mean position of the population. Instead, the dynamics of a typical agent are subject to congestion effects in the spirit of~\cite{achdouporretta2018}: it requires more effort to move in a crowded region than in a non-crowded region. To be consistent with the above theoretical analysis, we consider that the domain is the one-dimensional torus $\mathbb{T} = \mathbb{R} / \mathbb{Z}$. The two targets are located at $0.3$ and $0.7$. The Hamiltonian is:
\begin{align*}
	H(m,Du)
	=&\sup_{q\in\mathbb{R}^d}\left\{q\cdot Du- \frac12(1+4m)^{\beta}\vert q\vert ^2 -\zeta m\right\}\\
	=&\frac{1}{2(1+4m)^{\beta}}\vert Du\vert ^2-\zeta m\,.
\end{align*}
Here, the argmax is given by $q^{*}(x,t)=\frac{Du(x,t)}{(1+4m(x,t))^{\beta}}$ in $Q$, and $\beta, \zeta$ are positive constants.
The last term corresponds to a crowd aversion cost which discourages the agents from being in a very crowded region (independently of whether they move or not). 
We take a uniform distribution over $[0.375,0.625]$ for $m_0$. 

The corresponding MFG PDE system is:
\begin{equation}\label{Example1}
\begin{cases}
-\partial_tu-0.05{ \Delta} u+\frac{1}{2(1+4m)^{\beta}}\vert Du\vert ^2-\zeta m=0 & \text{ in }Q\\
\partial_tm -0.05{\Delta}  m-\textrm{div}(\frac{mDu}{(1+4m)^{\beta}})=0 & \text{ in }Q\\
u_T(x)=10\min \{(x-0.3)^2,(x-0.7)^2\} & \text{ in }\mathbb{T}\\
m_0(x)=4 \text{ for } x\in [0.375,0.625],\,\,m_0(x)= 0  &\text{ otherwise }.
\end{cases}
\end{equation}

We implement the two policy iteration methods on a finite-difference approximation of the above PDE system. We fix a grid $\mathcal{G}$ on ${\mathbb{T}^d}$. Then, we denote by $U, M$ and ${\mathsf{Q}}$ the vectors on $\mathcal{G}$ approximating respectively the solution and the policy. We will use the symbol $\sharp$ to denote suitable discretizations of the linear differential operators at the grid nodes. Here we use uniform grids and the centered second order finite differences for the discrete Laplacian, whereas the Hamiltonian and the divergence term in the FP equation are both computed via the Engquist-Osher numerical flux for conservation laws as in \cite{ccg}. To be more precise, in the present example which is in dimension $d=1$, we consider a uniform discretization of ${\mathbb{T}^d}$ with $I$ nodes $x_i = i \, h$, for $i=0,\dots,I-1$, where $h=1/I$ is the space step. We then introduce the discrete operators:
\begin{align*}
	({\Delta}_\sharp U)_i &=\frac{1}{h^2}\left(U_{[i-1]}-2U_i+U_{[i+1]}\right)\,,
	\\
	(D_\sharp U)_i &=\left(D_L U_i\,,\,D_R U_i\right)=\frac{1}{h}\left( U_i-U_{[i-1]}\,,\,U_{[i+1]}-U_i\right)\,,
\end{align*}
where the index operator $[\cdot]=\left\{(\cdot +I)\,mod\, I\right\}$ accounts for the periodic boundary conditions.  When updating the policy, we have
$$
	{\mathsf{Q}}_i=\left({\mathsf{Q}}_{i,L}\,,\,{\mathsf{Q}}_{i,R}\right)=\frac{1}{(1+4M_i)^{\beta}}\left( D_L U_i\,,\,D_R U_i\right)\,.
$$
Using the notation $(\cdot)^+=\max\left\{\cdot,0\right\}$ and $(\cdot)^-=\min\left\{\cdot,0\right\}$ for the positive and negative part respectively, we denote ${\mathsf{Q}}_\pm=({\mathsf{Q}}_L^+,{\mathsf{Q}}_R^-)$, and we have
$$
	(\vert {\mathsf{Q}}_\pm\vert ^2)_i=\left( {\mathsf{Q}}_{i,L}^+\right)^2+\left({\mathsf{Q}}_{i,R}^-\right)^2\,.
$$
The discrete divergence operator is such that:
\begin{align*}
	\left(\textrm{div}_\sharp(M\,{\mathsf{Q}})\right)_i=&\frac{1}{h}
	\left( M_{[i+1]}{\mathsf{Q}}_{[i+1],L}^+ - M_i {\mathsf{Q}}_{i,L}^+ \right)\\
	+& \frac{1}{h}\left(
	M_i {\mathsf{Q}}_{i,R}^- -M_{[i-1]} {\mathsf{Q}}_{[i-1],R}^-
	\right)\,.
\end{align*}

For the time discretization, we employ an implicit Euler method for both the time-forward FP equation and the time-backward HJB equation. To this end, we introduce a uniform grid on the interval $[0,T]$ with $N+1$ nodes $t_n=n \,{\Delta} t$, for $n=0,\dots, N$, and time step ${\Delta} t=T/N$. Then, we denote by $U_n, M_n$ and ${\mathsf{Q}}_n$ the vectors on $\mathcal{G}$ approximating respectively the solution and the policy at time $t_n$. In particular, we set on $\mathcal{G}$ the initial condition $M_{0,i}=m_0(x_i) / \sum_{j} h m_0(x_j)$ and the final condition $U_N=u_T(\cdot)$. 

The policy iteration algorithm \textbf{(PI1)} for the fully discretized system is the following: Given an initial guess ${\mathsf{Q}}^{(0)}_n:\mathcal{G}\to\mathbb{R}^{2d}$ for $n=0,\dots, N-1$, initial and final data $M_0,\,U_N:\mathcal{G}\to\mathbb{R}$, iterate on $k\geq 0$:

\begin{itemize}
	\item[(i)]  Solve on $\mathcal{G}$
	$$
	\left\{
	\begin{array}{l}
	 M^{(k)}_{n+1}-{\Delta} t\left(0.05{\Delta}_\sharp M^{(k)}_{n+1}+\textrm{div}_\sharp(M^{(k)}_{n+1}\,{\mathsf{Q}}^{(k)}_{n})\right)=M^{(k)}_{n}, \quad n=0,\dots, N-1\\
	 M^{(k)}_{0}=M_{0}    
	\end{array}
	\right.
	$$
	\item[(ii)] Solve on $\mathcal{G}$
	$$
			\left\{
	\begin{array}{l}
U_{n}^{(k)}- {\Delta} t\left(	 0.05{\Delta}_\sharp U^{(k)}_{n}-{\mathsf{Q}}^{(k)}_{n,\pm}\cdot D_\sharp U^{(k)}_{n}\right)\\
\hskip 19pt= U_{n+1}^{(k)}+{ \Delta} t\left(\frac12(1+4M^{(k)}_{n+1})^{\beta}\vert {\mathsf{Q}}^{(k)}_{n,\pm}\vert ^2+\zeta M^{(k)}_{n+1}\right), \quad n=0,\dots, N-1
\\
	 U^{(k)}_{N}=U_{N}    
	\end{array}
	\right.
	  $$
	\item[(iii)] Update the policy
	${\mathsf{Q}}^{(k+1)}_n=\frac{D_\sharp U^{(k)}_n}{(1+4M^{(k)}_{n+1})^{\beta}}$ on $\mathcal{G}$ for $n=0,\cdots, N-1$. %
\end{itemize}

Consistently with our theoretical convergence result, in the implementation we do not put any bound on the control. For the following results, we take $\beta=1.5$, $\zeta=1$ and $T = 1$ for the final horizon. We used a number of nodes $I=200$ in space and $N=200$ in time. The initial policy was set to 
${\mathsf{Q}}^{(0)}_n\equiv(0,0)$ on $\mathcal{G}$ for all $n$. Here we present results for \textbf{(PI1)}. %
\par
 In Figure~\ref{Fig1}, we report the time evolution of the density, by plotting, for several fixed $n$, the solution density $M_n$ and the policy ${\mathsf{Q}}_n={\mathsf{Q}}_{n,L}^+ +{\mathsf{Q}}_{n,R}^-$. We can see that the distribution splits into two parts, one moving towards the left target and one moving towards the right target. However, due to the congestion cost as well as the crowd aversion cost, each part can not concentrate exactly on the target.  \par
 In Figure~\ref{Fig2}, we report results on the convergence with respect to the number of iterations: the residuals of the discrete MFG system, as well as the discrete $L^\infty$ distance between ${\mathsf{Q}}^{(k)}, M^{(k)}$ and $U^{(k)}$ computed by the policy iteration and the final solution ${\mathsf{Q}}^{*}, M^{*}$ and $U^{*}$ from the fixed point iteration algorithm. Here we use the  fixed point iteration method to obtain a benchmark solution. This algorithm has been previously used for solving mean field games with non-separable Hamiltonians or mean field type control problems e.g. in~\cite{achdou2020}. The details are provided below. Here we observe that the solution via our policy iteration method is consistent with the benchmark solution. Moreover, consistently with our theoretical findings (see Section~\ref{sec:estimate}), the convergence rate is linear, except for the first few iterations and after a lower bound is reached due to the limitations on the approximation accuracy of the discrete system.
\begin{figure}[h!]
\begin{center}
\begin{tabular}{cc}
\includegraphics[width=.45\textwidth]{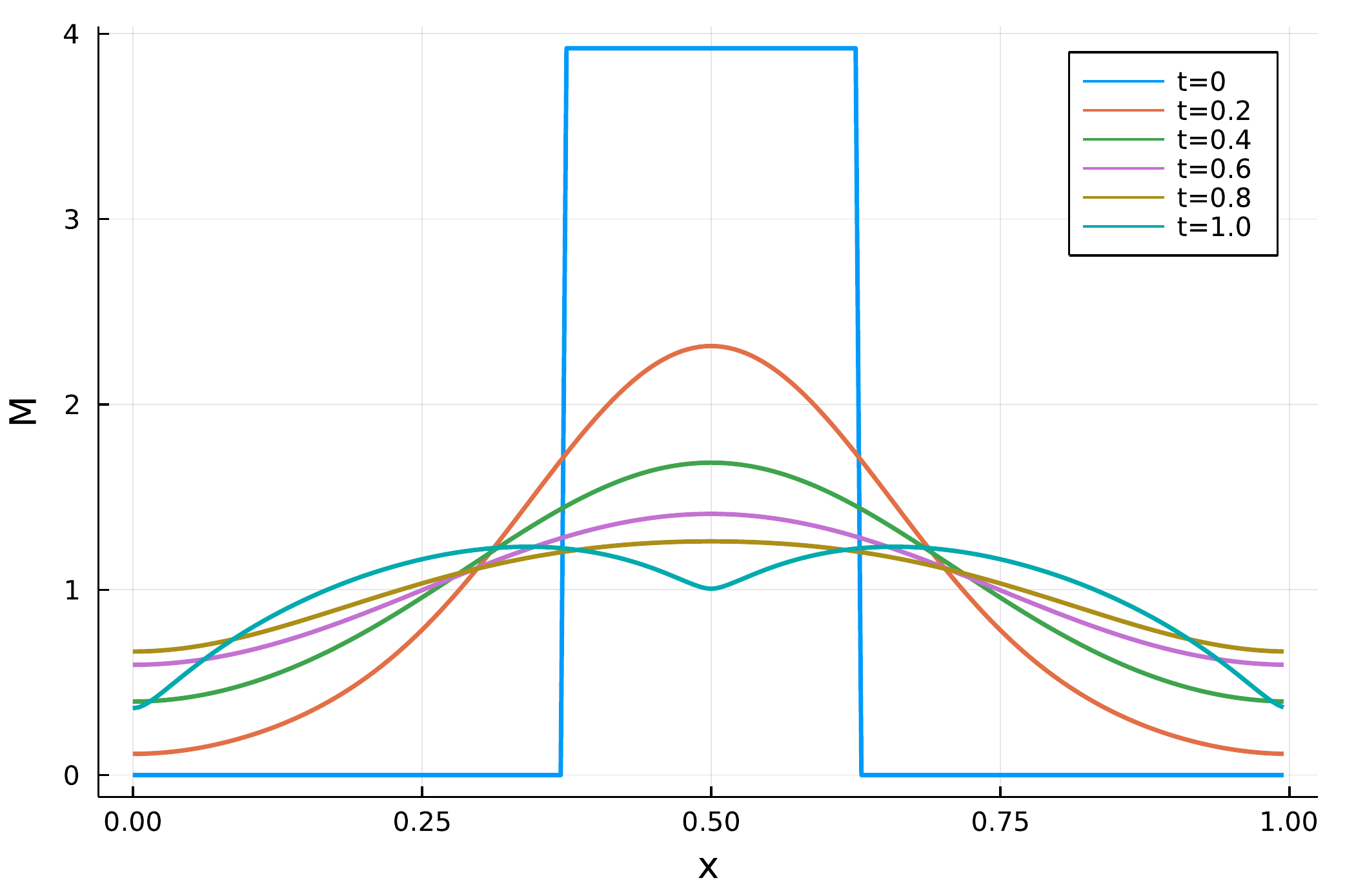} &
\includegraphics[width=.45\textwidth]{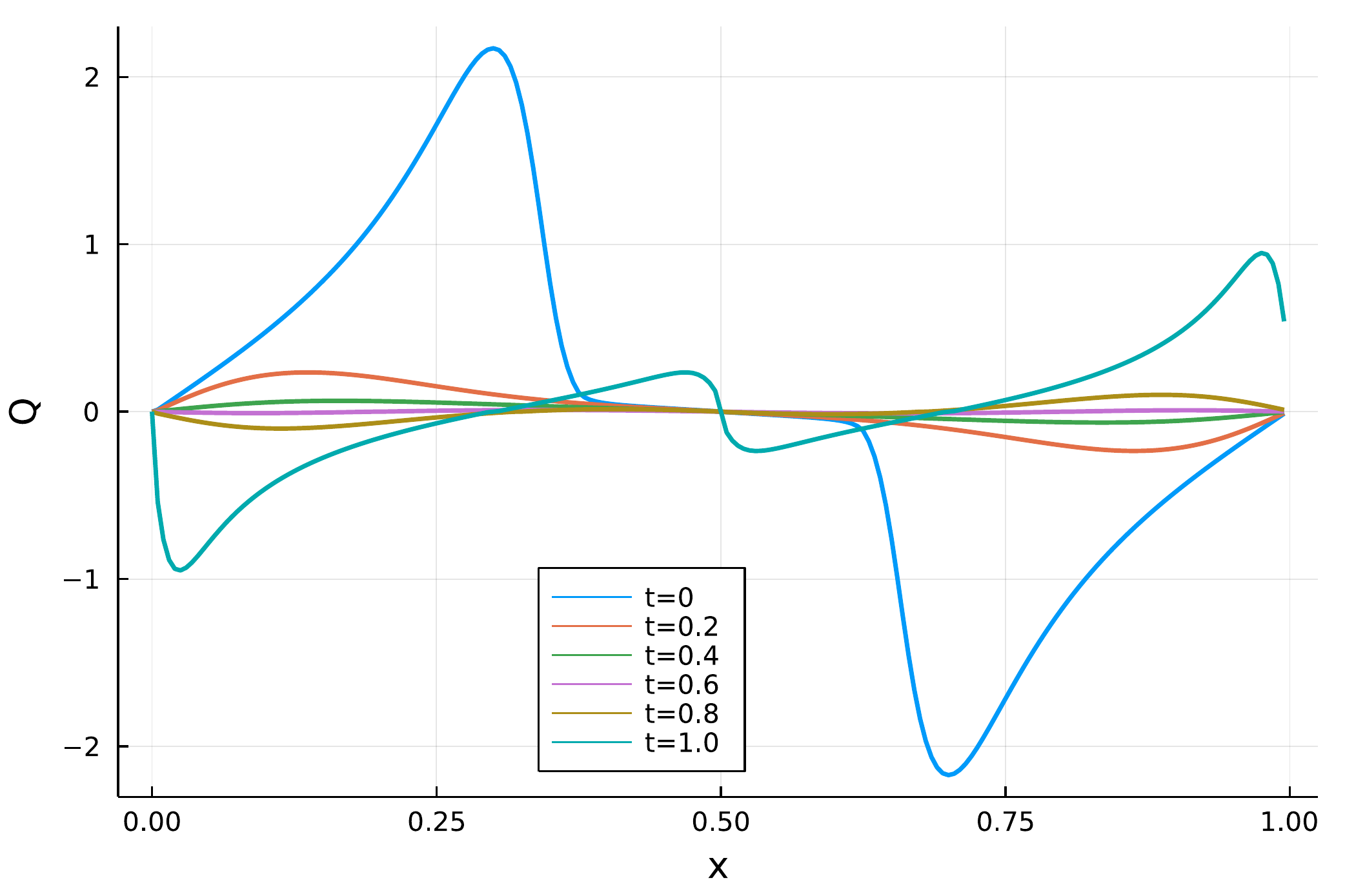} \\
(a)&(b)
\end{tabular}
\end{center}
\caption{Example 1. Solution for the MFG system~\eqref{Example1} obtained with policy iteration \textbf{(PI1)}. (a) The density $M$ and  (b) the policy ${\mathsf{Q}}={\mathsf{Q}}^+_L+{\mathsf{Q}}^-_R$ at several time steps.}\label{Fig1}
\end{figure}

\begin{figure}[h!]
\begin{center}
\begin{tabular}{cc}
\includegraphics[width=.45\textwidth]{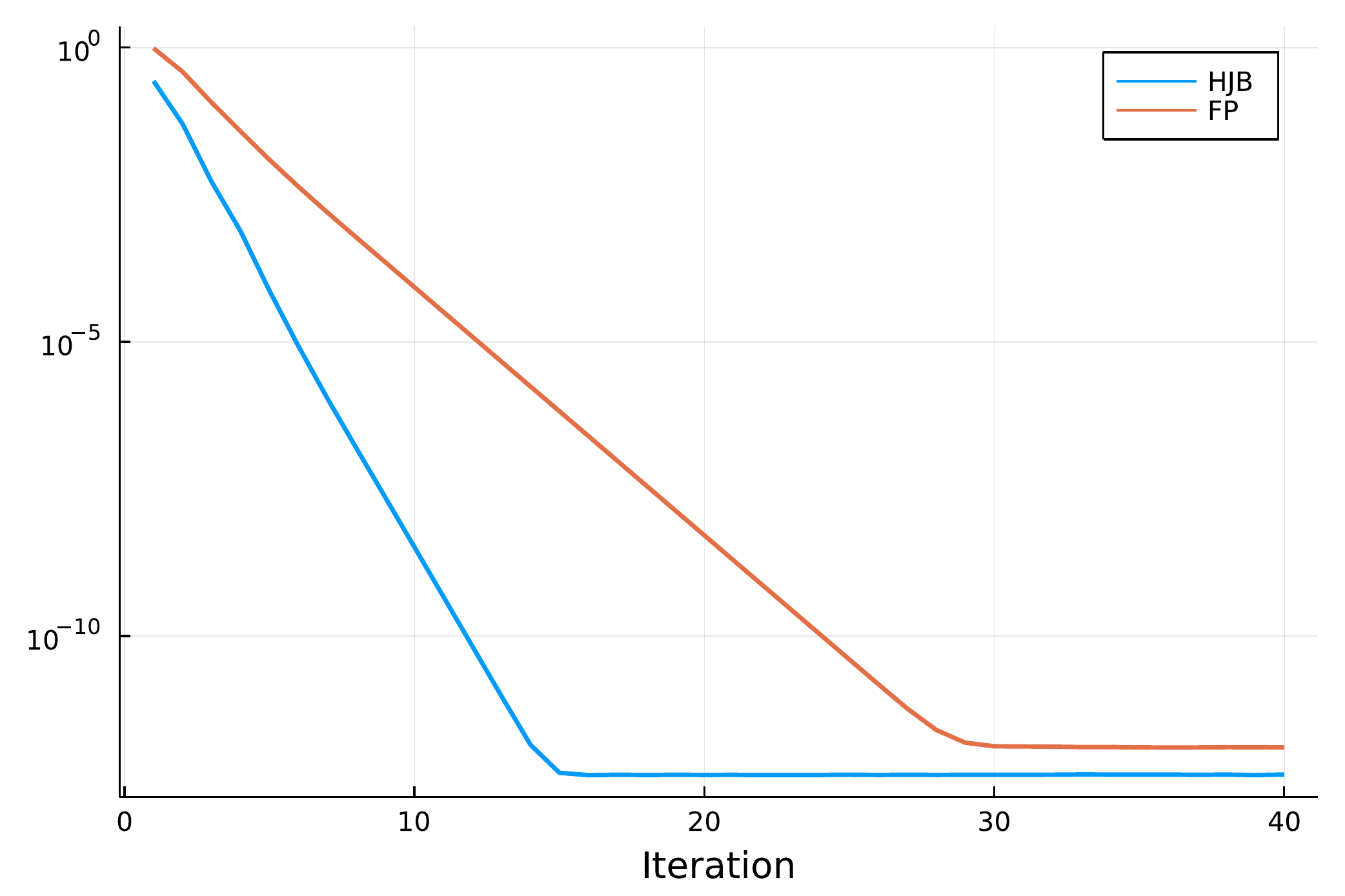} &
\includegraphics[width=.45\textwidth]{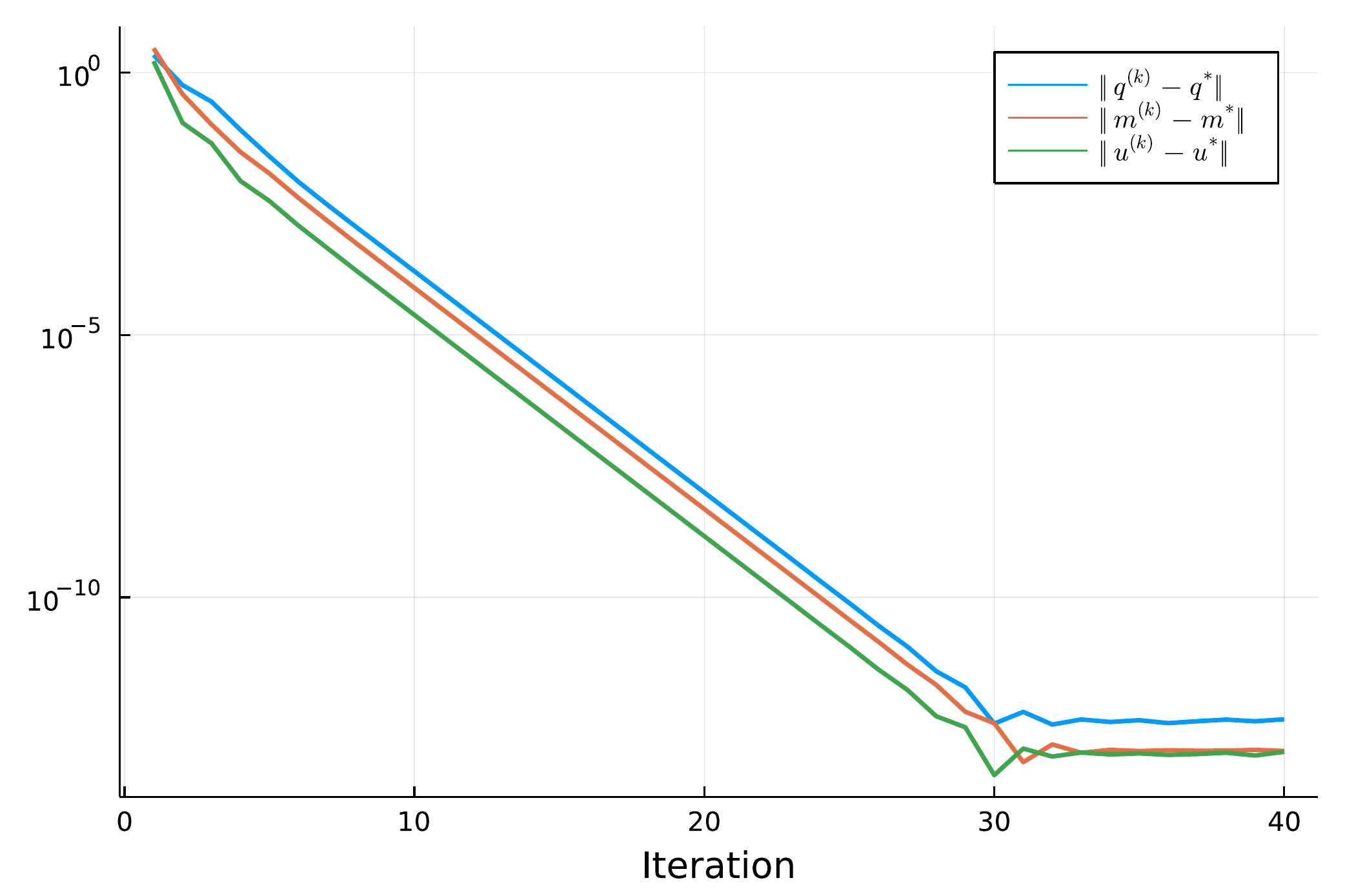} \\
(a)&(b)
\end{tabular}
\end{center}
\caption{Example 1. Convergence of policy iteration \textbf{(PI1)} for the MFG system~\eqref{Example1}. (a) Residuals of MFG system, (b) $L^\infty$ distance between ${\mathsf{Q}}^{(k)}, M^{(k)}$ and $U^{(k)}$ from policy iteration and the final solution ${\mathsf{Q}}^{*}, M^{*}$ and $U^{*}$ from fixed point algorithm.}\label{Fig2}
\end{figure}

We now give the details of the fixed point iteration algorithm for solving the discrete MFG system with a non-separable Hamiltonian. It iterates over the distribution and the value function. The main difference with the policy iteration method \textbf{(PI1)} is that at each iteration, we solve an HJB equation instead of solving a linear equation with a given control. As explained below, for this step we rely on Newton method. \par
The main iteration of the fixed point method is the following outer loop: Given an initial guess $M^{(0)}_n:\mathcal{G}\to\mathbb{R}^{2d}$ for $n=1,\cdots, N$, and $U^{(0)}_n:\mathcal{G}\to\mathbb{R}^{2d}$ for $n=0,\dots, N-1$, iterate on $k\geq 1$ up to convergence,

\begin{itemize}
	\item[(i)]  Solve on $\mathcal{G}$:
	$$
	\left\{
	\begin{array}{l}
	 M^{(k)}_{0}=M_{0}    
	 \\
	 \frac{M^{(k)}_{n+1}-M^{(k)}_{n}}{{\Delta} t}-0.05{\Delta}_\sharp M^{(k)}_{n+1}-\textrm{div}_\sharp(M^{(k)}_{n+1}\,\frac{DU^{(k-1)}_{n}}{(1+4M^{(k-1)}_{n+1})^{\beta}})=0, 
	 \\ \hfill \quad n=0,\dots, N-1
	\end{array}
	\right.
	$$
	\item[(ii)] Solve on $\mathcal{G}$:
	$$
			\left\{
	\begin{array}{l}
	 U^{(k)}_{N}=U_{N}    
	 \\
		-\frac{U_{n+1}^{(k)}-U_{n}^{(k)}}{{ \Delta} t}- 0.05{\Delta}_\sharp U^{(k)}_{n}+\frac{\vert D_\sharp U^{(k)}_{n}\vert ^2}{2(1+4M^{(k)}_{n+1})^{\beta}} - \zeta M^{(k)}_{n+1} = 0, 
	\\ \hfill \quad n=N-1,\dots, 0.
	\end{array}
	\right.
	  $$
	
\end{itemize}

We use a forward time marching method and backward time marching method respectively for step (i) and (ii). In step (ii) we need to solve a nonlinear system for every time step $n$. We do this by Newton method, which consists in the following inner loop. 
      For given $n$ and $k$, set initial guess $\tilde{U}^{(\tilde{k}=0)}_{n}=U^{(k-1)}_{n}$, and then iterate on $\tilde{k}\geq 0$:
\begin{itemize}
	\item[($ii_1$)]  Compute the residual of HJB system:
	$$
	\mathcal{F}^{(\tilde{k})}_n(\tilde{U}^{(\tilde{k})}_{n}) = -\frac{U_{n+1}^{(k)}-\tilde{U}^{(\tilde{k})}_{n}}{{\Delta} t}-0.05{\Delta}_\sharp \tilde{U}^{(\tilde{k})}_{n}+\frac{\vert D_\sharp\tilde{U}^{(\tilde{k})}_{n}\vert ^2}{2(1+4M^{(k)}_{n+1})^{\beta}}-\zeta M^{(k)}_{n+1},
	$$
	\item[($ii_2$)] Compute the Jacobian matrix:
	$$
	\mathcal{J}^{(\tilde{k})}_n(\tilde{U}^{(\tilde{k})}_{n}) = \frac{1}{{\Delta} t}I - 0.05{\Delta}_\sharp + \frac{D_\sharp\tilde{U}^{(\tilde{k})}_{n}}{(1+4M^{(k)}_{n+1})^{\beta}} \cdot D_\sharp,
	$$
	\item[($ii_3$)] Update:
	$\tilde{U}^{(\tilde{k}+1)}_{n} = \tilde{U}^{(\tilde{k})}_{n} + (\mathcal{J}^{(\tilde{k})}_n)^{-1} (-\mathcal{F}^{(\tilde{k})}_n)$. 
	\end{itemize}
	
For step ($ii_3$), instead of computing the inverse of the Jacobian matrix $\mathcal{J}^{(\tilde{k})}_n$, an alternative method is to first solve a linear system to find $(\tilde{U}^{(\tilde{k}+1)}_{n} - \tilde{U}^{(\tilde{k})}_{n})$ and then deduce $\tilde{U}^{(\tilde{k})}_{n}$ from here. 

The aforementioned benchmark solution ${\mathsf{Q}}^{*}, M^{*}$ and $U^{*}$ is obtained by running the fixed point method until convergence (up to numerical approximations). 

We perform  some numerical experiments on the maximal time horizon $T$ with which the algorithm converges. The results are listed in Table~\ref{table-maximumT} with different values of $\beta$ and $\zeta$, without changing the step size ${\Delta} t$, ${\Delta} x$ or any other parameters. In some cases the algorithm converges even when $T=50$. The fact that convergence depends heavily on the value of the constant  $\zeta$ is consistent with the theoretical findings of \cite{ct}.

\begin{table}[!h]
	\centering
	 \begin{tabular}{| c | c | c | c | c | }
	   \hline
		$\zeta \backslash \beta$ & 1.5 & 1.2 & 1.0 & 0.8\\
		\hline\hline
		0.8 & $> 50 $&$< 3.4$ &$ <1.6 $&$ <1.1$\\
		\hline
		0.6 & $> 50$ &$ > 50$ & $<3.0$ &$ <1.5$ \\
		\hline
		0.4 & $> 50$ & $> 50 $&$ > 50$&$ <4.2$ \\
		\hline
		0.2 & $> 50$ &$> 50$ &$> 50 $&$ > 50$ \\
		\hline
	\end{tabular}
\caption{Maximum $T$ with which policy iteration algorithm \textbf{(PI1)} converges, with different $\beta$ and $\zeta$}\label{table-maximumT}
\end{table}

\vskip 6pt
\noindent\textbf{Example 2:} We now give an example in dimension $d=2$ in which the domain is $\mathbb{T}^2$. The running cost represents congestion effects, but in this example the Hamiltonian $H(m,Du)$ is singular at $m=0$. There is a terminal cost that encourages the agents to move towards some sub-regions of the domain. The initial distribution is a truncated Gaussian distribution centered around $(0.25, 0.25)$. The MFG PDE system is:
\begin{equation}\label{example2}
\begin{cases}
	-\partial_tu-0.3{\Delta} u+\frac{1}{2m^{1/2}}\vert Du\vert ^2=0 & \text{ in }Q\\
	\partial_tm -0.3{\Delta}  m-\textrm{div}(\frac{mDu}{m^{1/2}})=0 & \text{ in }Q\\
	u_T(x_1,x_2)=1.2\cos(2\pi x_1)+\cos(2\pi x_2) & \text{ in }\mathbb{T}^2\\
	m_0(x_1,x_2)=\mathsf{C}\exp\{-10[(x_1-0.25)^2+(x_2-0.25)^2]\}  &\text{ in }\mathbb{T}^2,
\end{cases}
\end{equation}
where $\mathsf{C}$ is a constant such that $\int_{\mathbb{T}^2}m_0(x)dx=1$. We set the terminal time $T=0.5$. The finite-difference scheme described above can be adapted to this two-dimensional example in a straightforward way. See e.g.~\cite{ad,acd} for more details. For the numerical results provide below, we used $I=50$ nodes in each space dimension and $N=50$ nodes in time. 

We compare the two policy iteration methods that we proposed, namely \textbf{(PI1)} and \textbf{(PI2)}. For both methods, we used the initial policy ${\mathsf{Q}}^{(0)}_n\equiv(0,0,0,0)$ on $\mathcal{G}$ for $n=0,\dots, N-1$. In Figure~\ref{Fig3}, we give residuals of MFG system and the discrete $L^\infty$ distance between ${\mathsf{Q}}^{(k)}, M^{(k)}$ and $U^{(k)}$ at each iteration and the final solution ${\mathsf{Q}}^{*}, M^{*}$ and $U^{*}$ from the fixed point iteration algorithm. We see that it takes about 37 iterations with \textbf{(PI1)} to decrease $\displaystyle\max_{n,i,j}\vert  M^{(k+1)}_{n,i,j}-M^{(k)}_{n,i,j}\vert $ to $ 10^{-8}$ whereas it takes only 29 iterations with algorithm \textbf{(PI2)}. In Figure~\ref{Fig4},  we report the contours of density $M_n$ at different time $t$ with the algorithm \textbf{(PI2)}. Since both methods yield similar results, we omit the contours of $M_n$ obtained with the algorithm \textbf{(PI2)}.

\begin{figure}[h!]
\begin{center}
\begin{tabular}{cc}
\includegraphics[width=.45\textwidth]{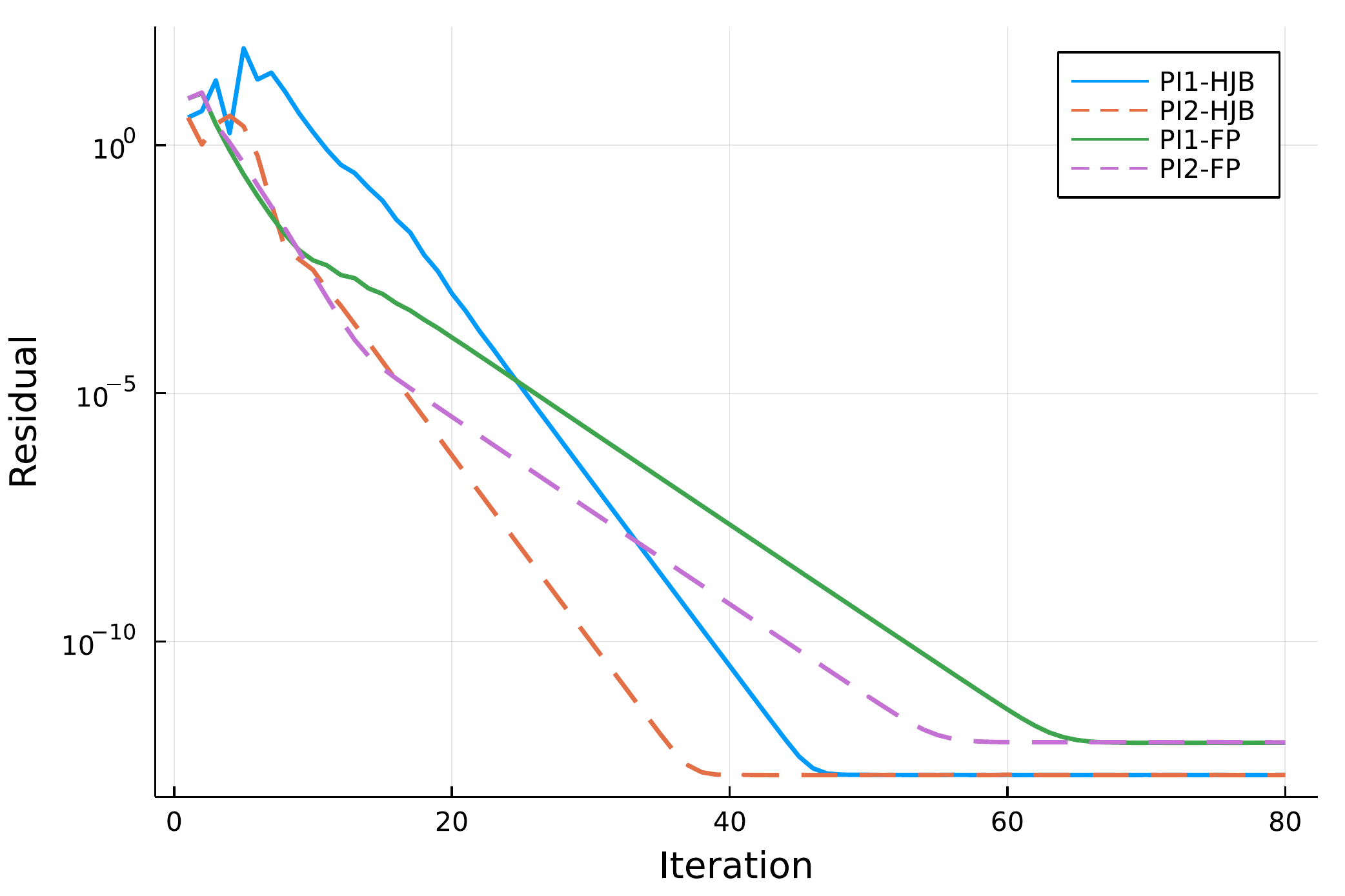} &
\includegraphics[width=.45\textwidth]{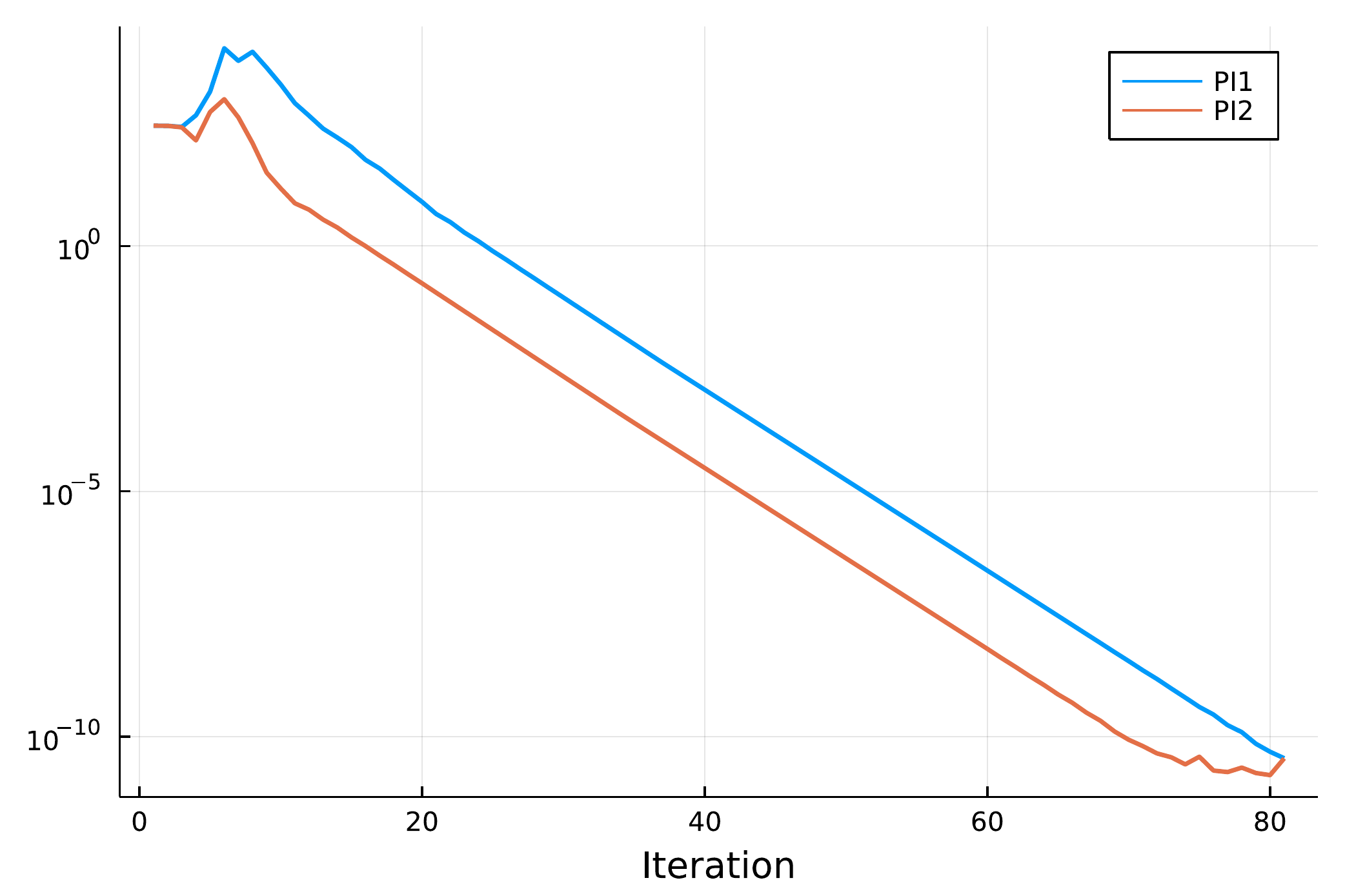} \\
residual & $\Vert q^{(k)}-q^{*} \Vert$ \\\\
\includegraphics[width=.45\textwidth]{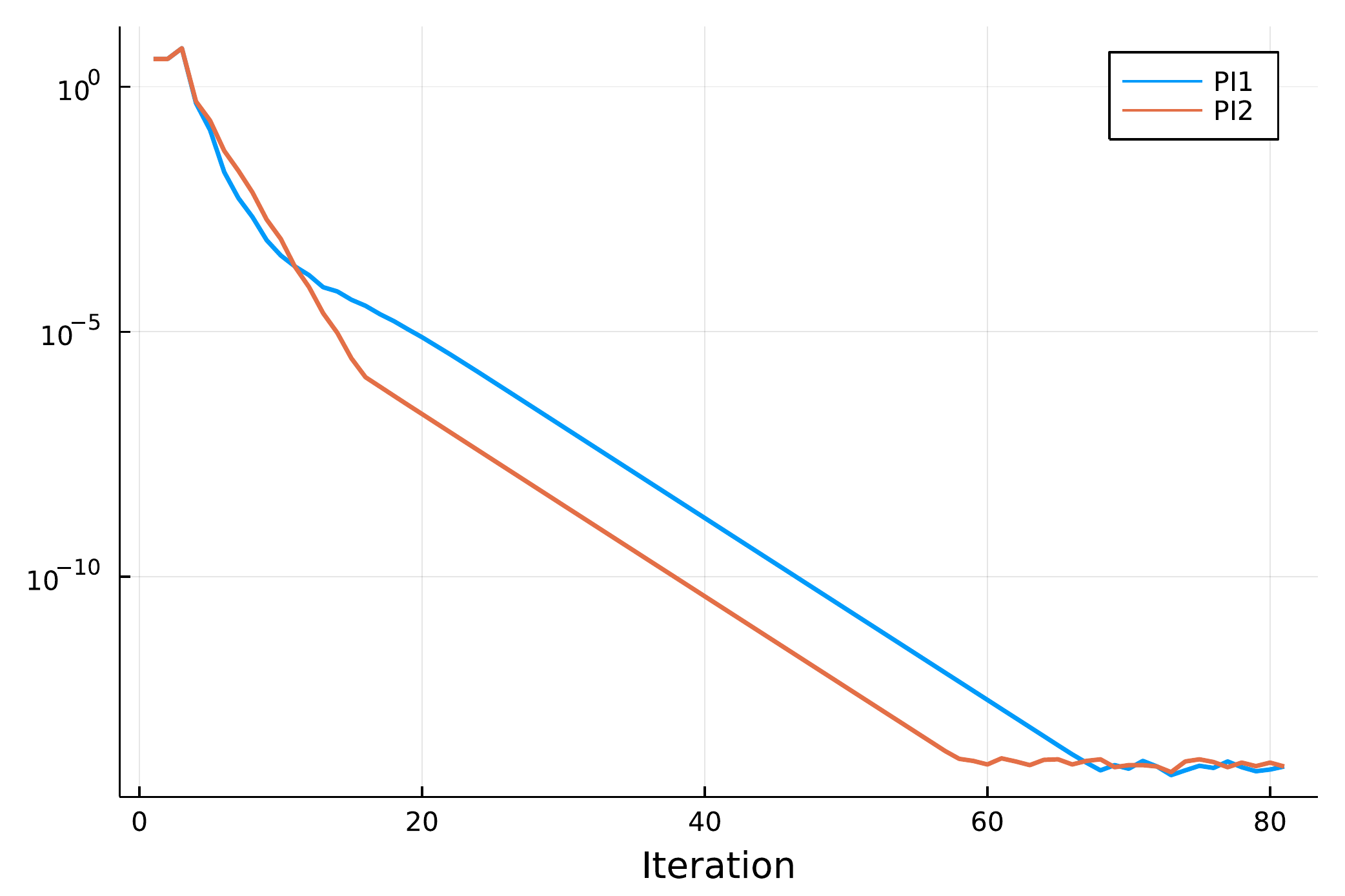} &
\includegraphics[width=.45\textwidth]{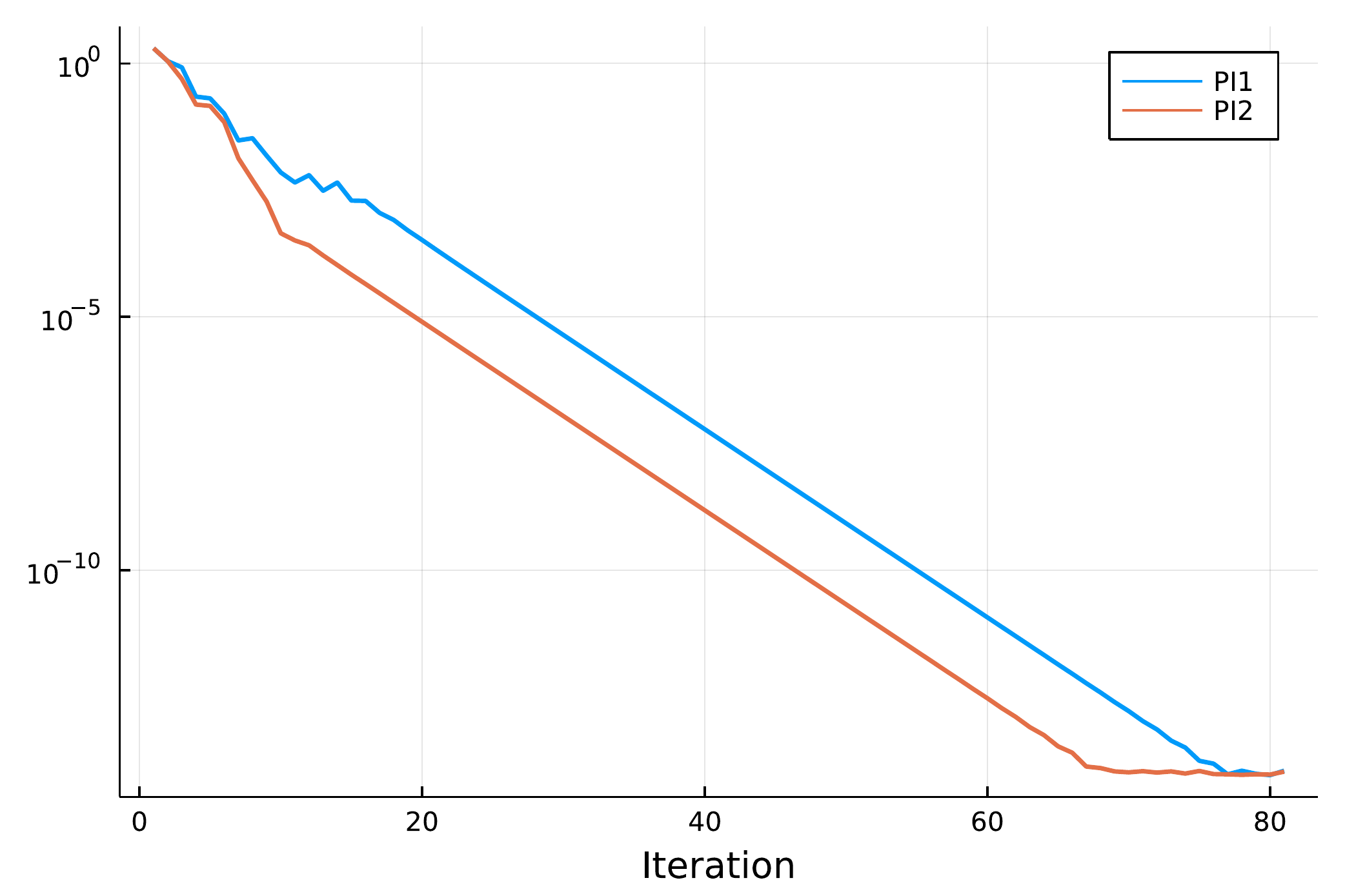} \\
 $\Vert m^{(k)}-m^{*} \Vert$ & $\Vert u^{(k)}-u^{*} \Vert$\\\\
\end{tabular}
\end{center}
\caption{Example 2. The residual of MFG system and the $L^\infty$ distance between ${\mathsf{Q}}^{(k)}, M^{(k)}, U^{(k)}$ from policy iteration (\textbf{(PI1)} or \textbf{(PI2)}) and the final solution ${\mathsf{Q}}^{*}, M^{*}, U^{*}$ from fixed point iteration algorithm.}\label{Fig3}
\end{figure}

\begin{figure}[h!]
\begin{center}
\begin{tabular}{cc}
\includegraphics[width=.4\textwidth]{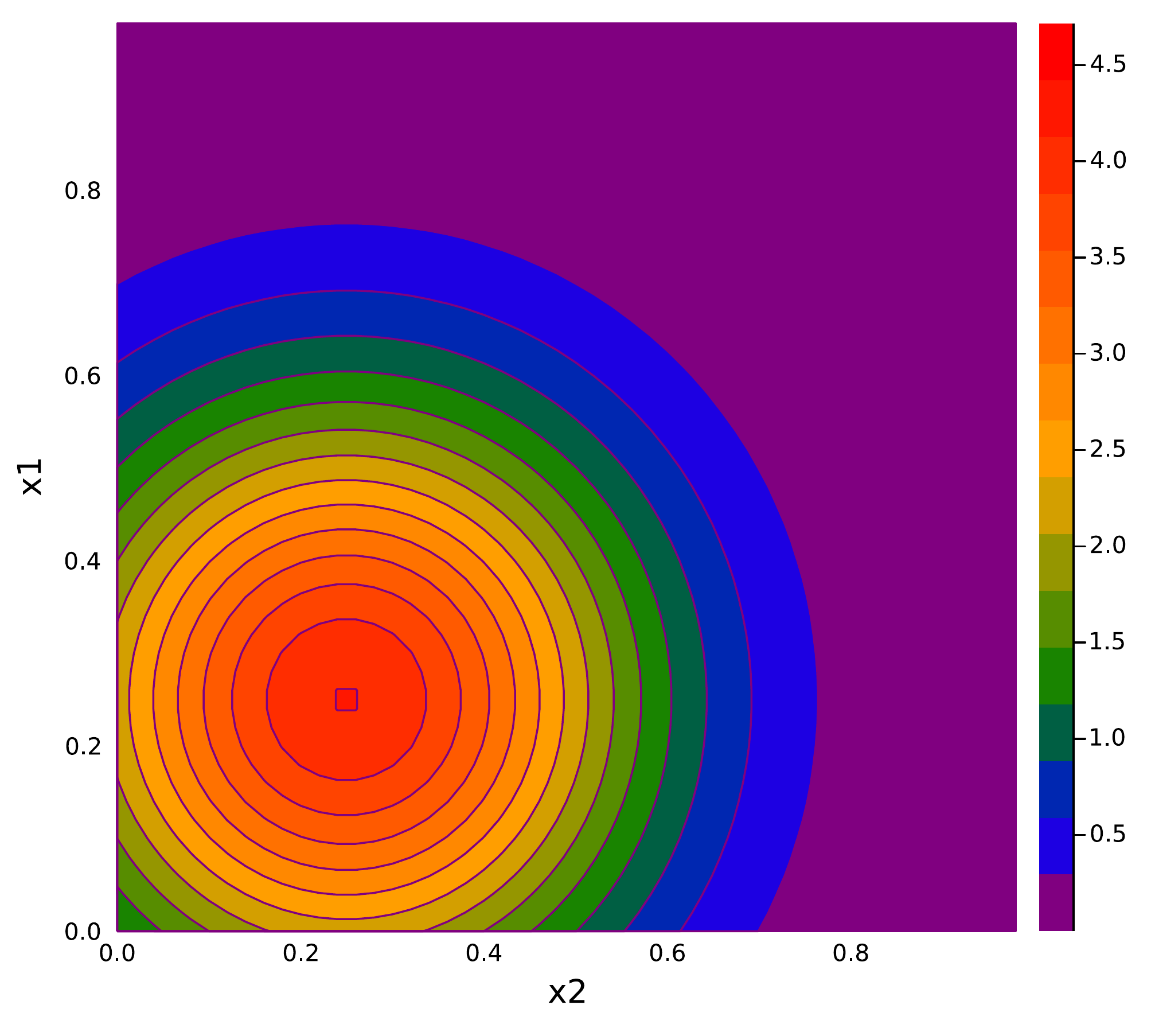} &
\includegraphics[width=.4\textwidth]{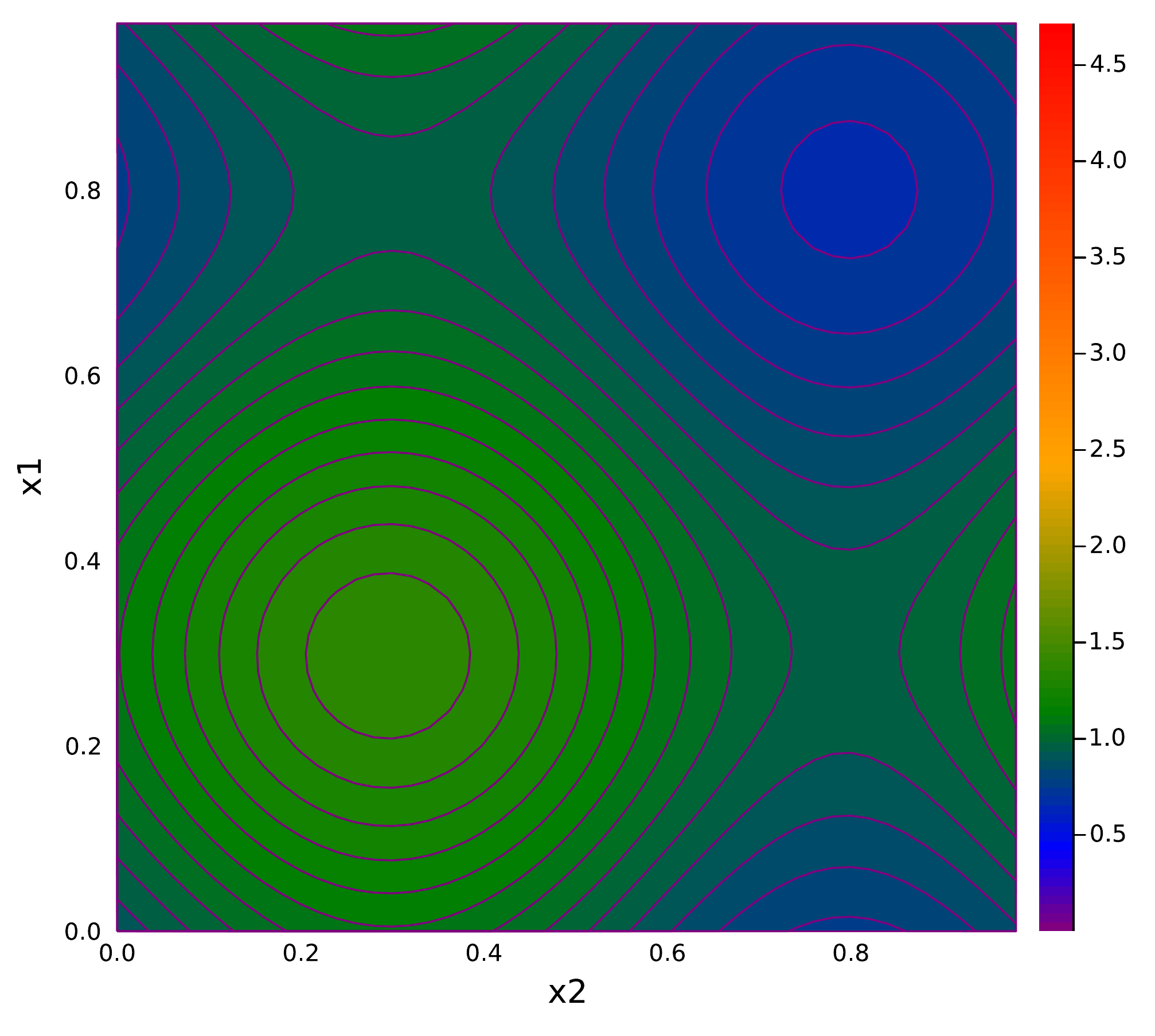} \\
$t=0$ & $t=0.16$ \\\\
\includegraphics[width=.4\textwidth]{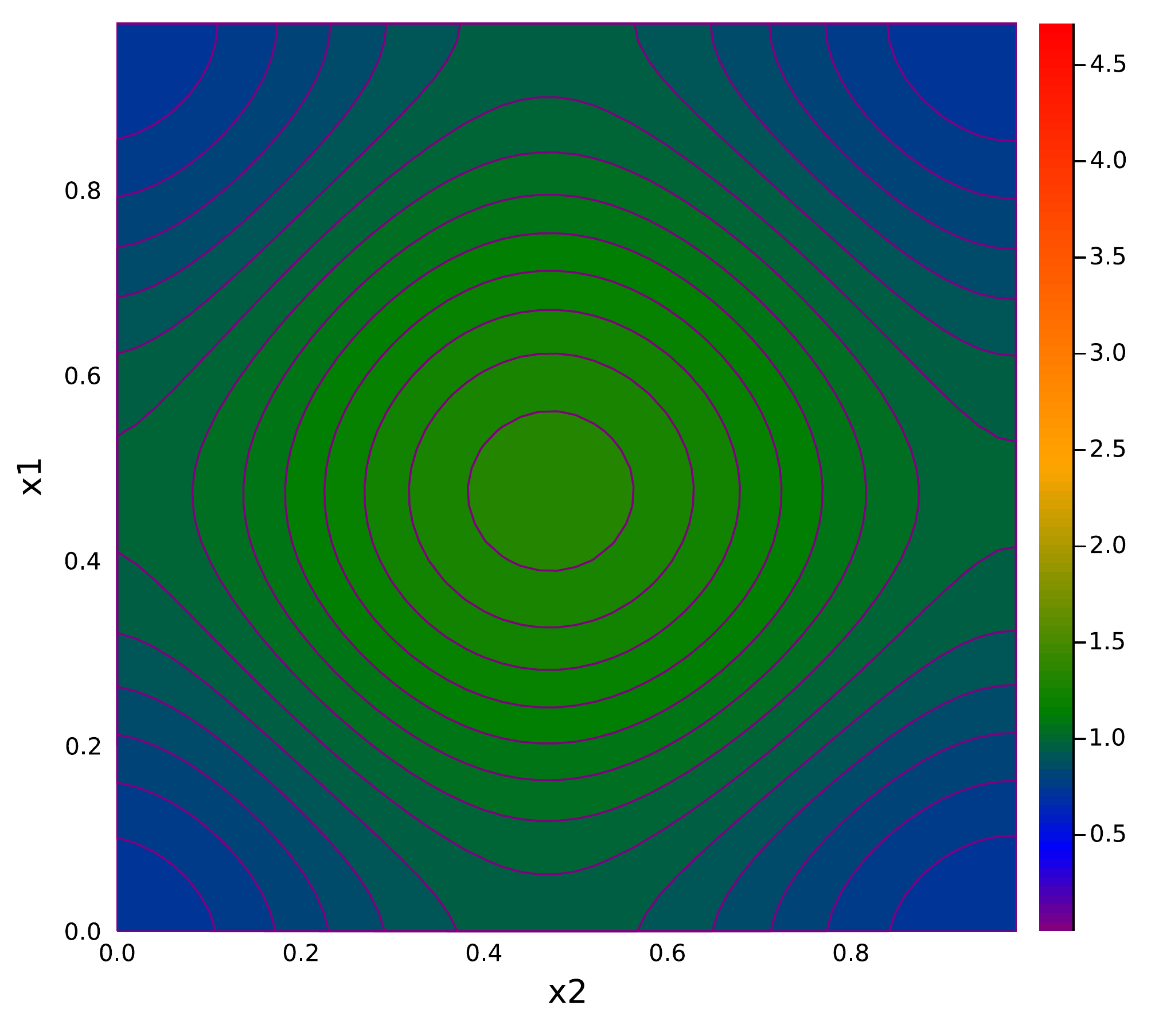} &
\includegraphics[width=.4\textwidth]{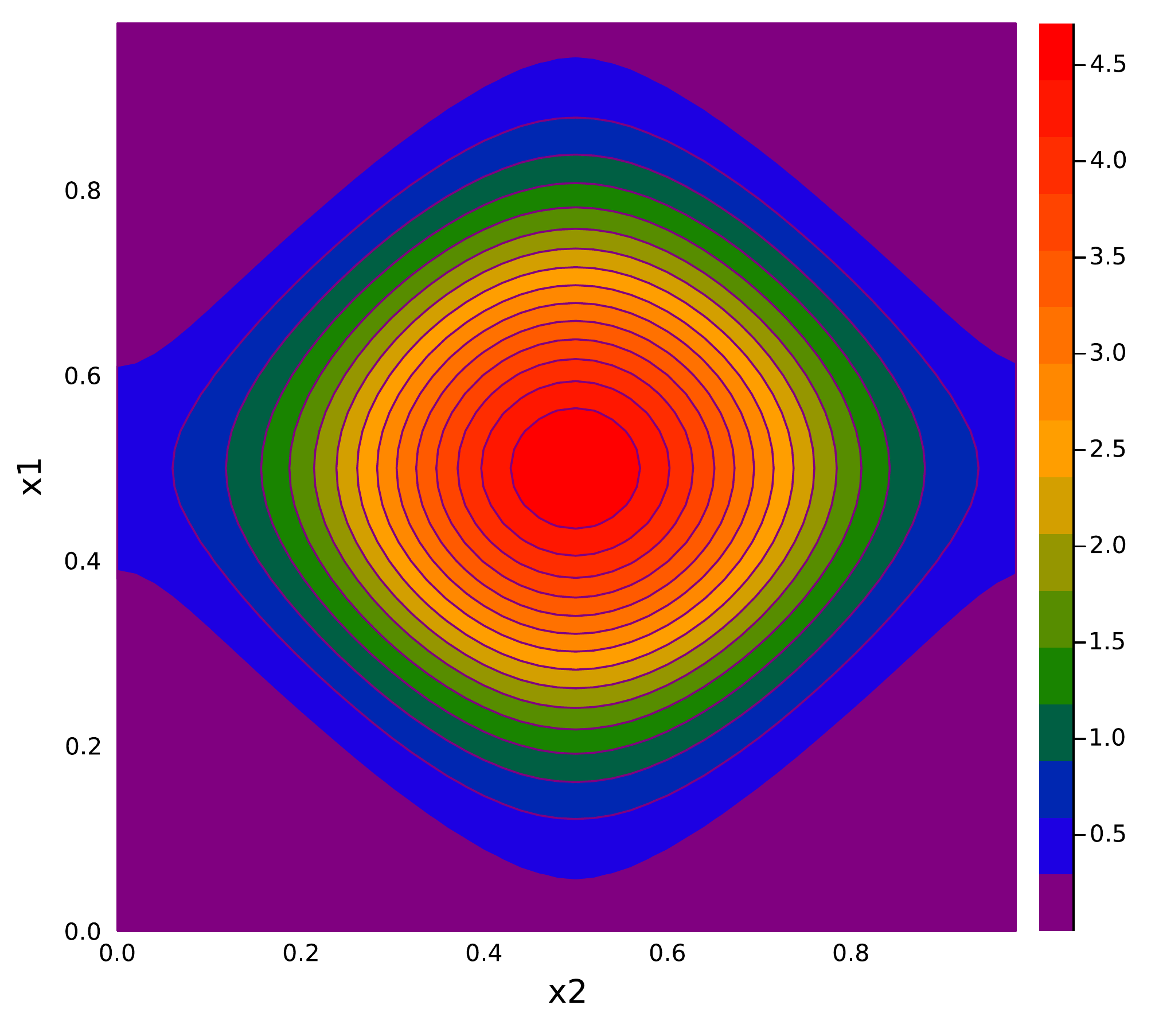} \\
$t=0.33$ & $t=0.5$
\end{tabular}
\end{center}
\caption{Example 2. Solution obtained by \textbf{(PI2)} for the 2d MFG system~\eqref{example2} : contours of density at several time steps.}\label{Fig4}
\end{figure}

Using the same setting as for the policy iteration algorithm, and the initial guess $U^{(0)}_n\equiv 0, M^{(0)}_n\equiv 1$, the fixed point algorithm converges with 27 (outer) iterations. In Figure~\ref{Fig5}, we report the residual of discrete MFG system with the fixed point iteration and \textbf{(PI2)}. For the latter, we count the number of outer iterations. In the fixed point iteration algorithm, at iteration $k$ we have a fixed $M^{(k)}$ and solve the HJB equation using Newton method, hence the HJB residual is very small. We see that the residual for the FP equation is slightly smaller than the one obtained with \textbf{(PI2)}, but it roughly decays at the same rate. 
However, remember that here we are comparing iterations of \textbf{(PI2)} with \textit{outer} iterations of the fixed point method, but each iteration of the latter involves an inner loop for the Newton method. Furthermore there is no clear way to parallelize this inner loop. As a consequence, the fixed point method is overall more expensive from a computational viewpoint. For the sake of illustration, we provide in Table~\ref{table-policyvsdirect} computational times obtained with each method on a computer with Intel(R) Xeon(R) processor running at 2.20GHz. Note that after $60$ iterations, the fixed point method has basically converged and hence the inner loop with Newton method converges much faster than during the first iterations because the initial guess for the non-linear HJB equation is already quite correct.   

\begin{table}[!h]
	\centering
	 \begin{tabular}{| c | c | c | }
	   \hline
		Iterations & \textbf{(PI2)} Total CPU (secs) & Fixed Point Total CPU (secs)\\
		\hline\hline
		10 & 18.82 & 33.51 \\
		\hline
		20 & 37.23 & 55.42 \\
		\hline
		30 & 56.59 & 77.30 \\
		\hline
		60 & 114.43 & 132.41 \\
		\hline
	\end{tabular}
\caption{Policy iteration \textbf{(PI2)} vs Fixed Point iteration total CPU times with different number of iterations. }\label{table-policyvsdirect}
\end{table}

\begin{figure}[h!]
\begin{center}
\begin{tabular}{cc}
\includegraphics[width=.45\textwidth]{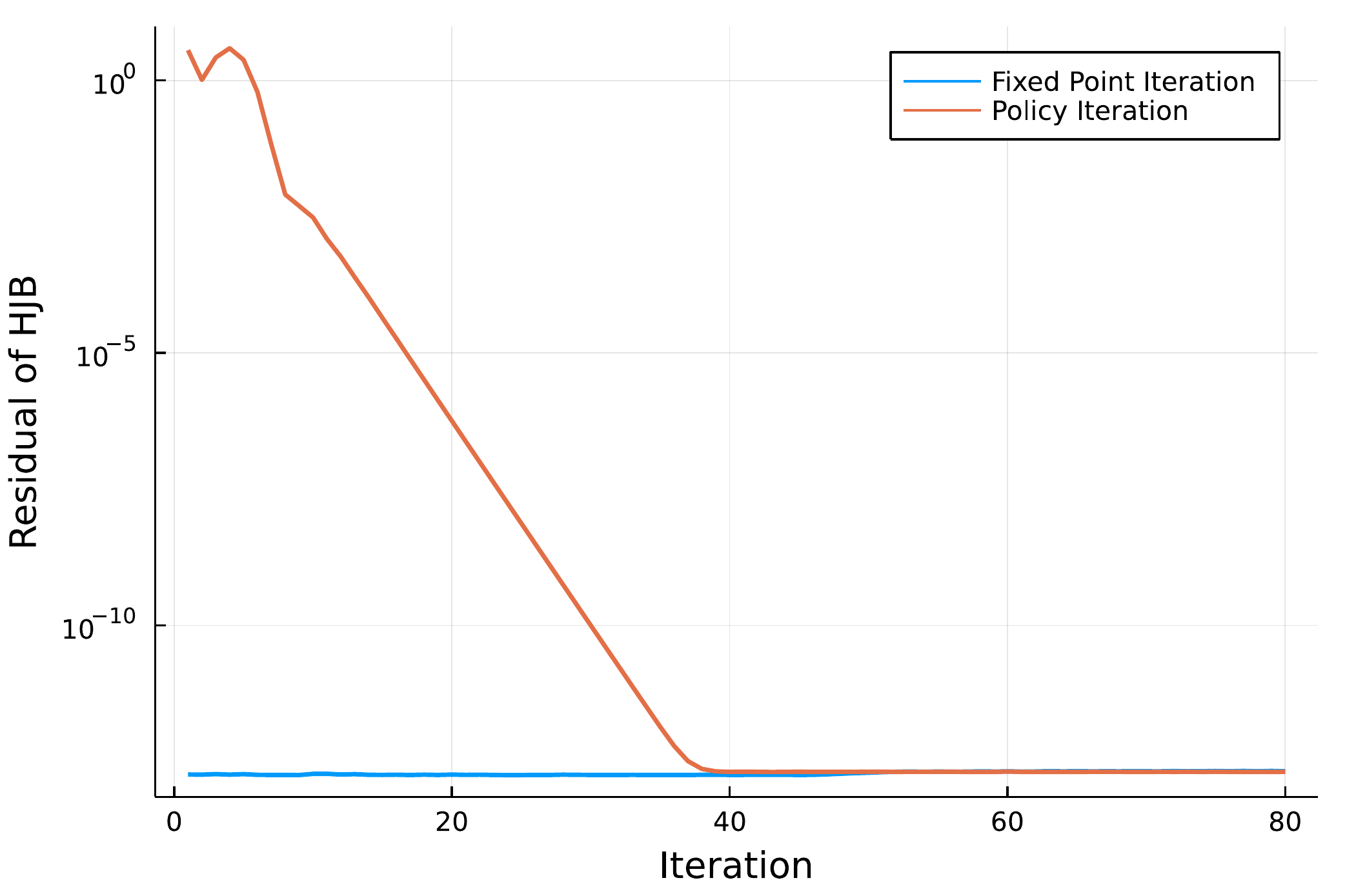} &
\includegraphics[width=.45\textwidth]{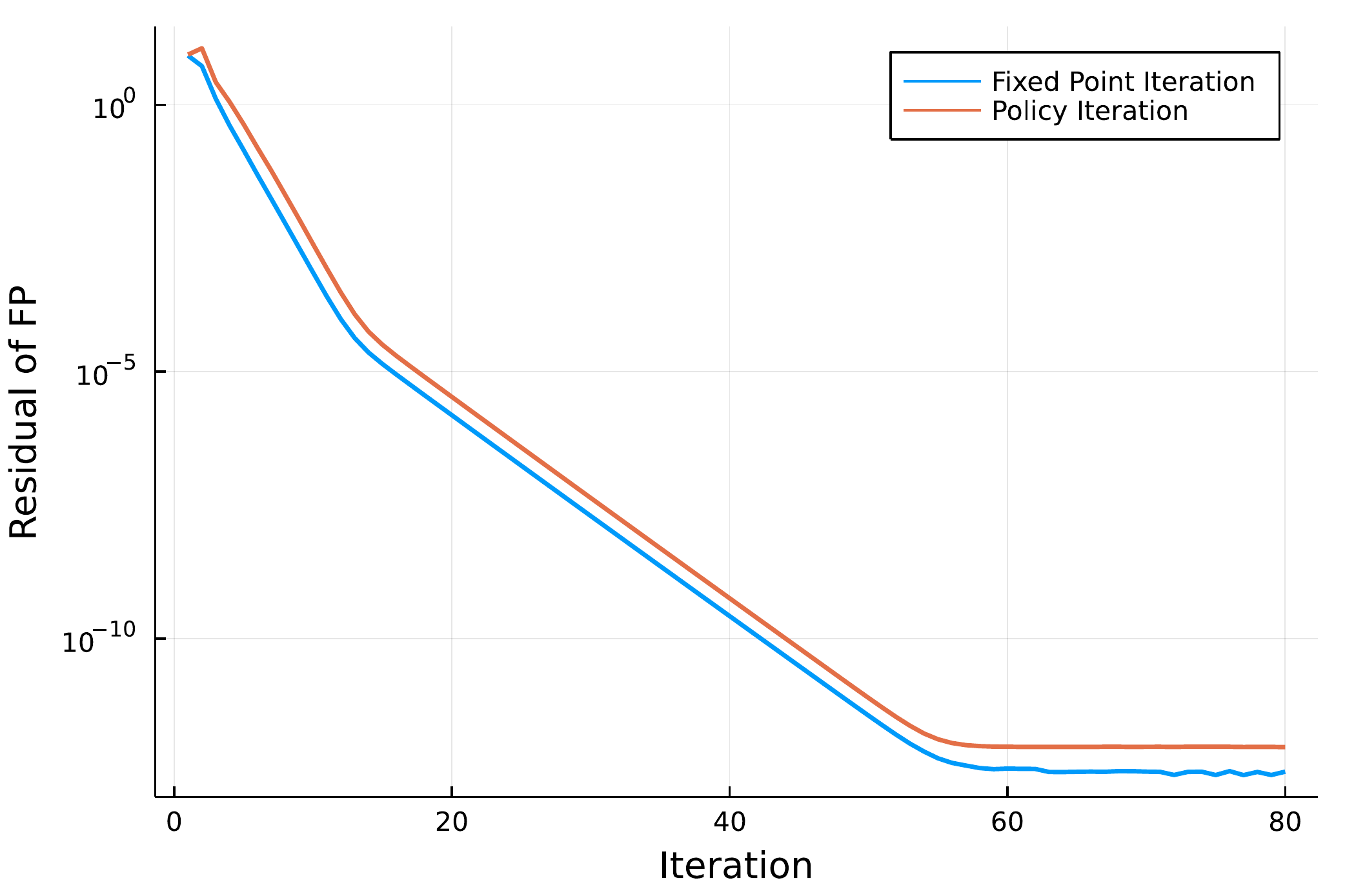} \\
(a)&(b)
\end{tabular}
\end{center}
\caption{Example 2. (a) the residual of HJB equation, (b) the residual of FP equation.}\label{Fig5}
\end{figure}

\vskip 6pt
\noindent\textbf{Example 3:} We conclude with a variant of Example 2, where we take a super-quadratic nonlinearity for the gradient term in the Hamiltonian. Note that our theoretical results also apply to this setting.  
We take the following Hamiltonian:
\begin{equation*}
	H(m,Du) = \max_q\left\{qDu - \frac23 m^{1/4} \vert q\vert ^\frac32\right\}  = \frac{\vert Du\vert ^3}{3m^{1/2}}, 
\end{equation*}
where the maximizer is: $q(x,t)=\frac{\vert Du\vert }{m^{1/2}}Du$ in $Q$. 
The corresponding PDE system is:
\begin{equation}\label{example3}
\begin{cases}
-\partial_tu-0.3{\Delta} u+\frac{1}{3m^{1/2}}\vert Du\vert ^3=0 & \text{ in }Q\\
\partial_tm -0.3{\Delta}  m-{\rm{div}}(\frac{mDu\vert Du\vert }{m^{1/2}})=0 & \text{ in }Q\\
u_T(x_1,x_2)=1.2\cos(2\pi x_1)+\cos(2\pi x_2) & \text{ in }\mathbb{T}^2\\
m_0(x_1,x_2)=C\exp\{-10[(x_1-0.25)^2+(x_2-0.25)^2]\}  &\text{ in }\mathbb{T}^2\\
\end{cases}
\end{equation}
Using the same setting as in Example 2, the policy iteration algorithm \textbf{(PI1)} with 46 iterations leads to $\displaystyle\max_{n,i,j}\vert  M^{(k+1)}_{n,i,j}-M^{(k)}_{n,i,j}\vert $ smaller than $ 10^{-8}$. The contours of density $M_n$ at different time $t_n$ are displayed in Figure~\ref{Fig6}, which is to be compared with Figure~\ref{Fig4}. We see that in the present example, the mass is much more concentrated at the terminal time. This can be explained by the fact that a super-quadratic Hamiltonian corresponds to a lower congestion cost. Hence the agents can move faster and get closer to a desired position.

\begin{figure}[h!]
\begin{center}
\begin{tabular}{cc}
\includegraphics[width=.4\textwidth]{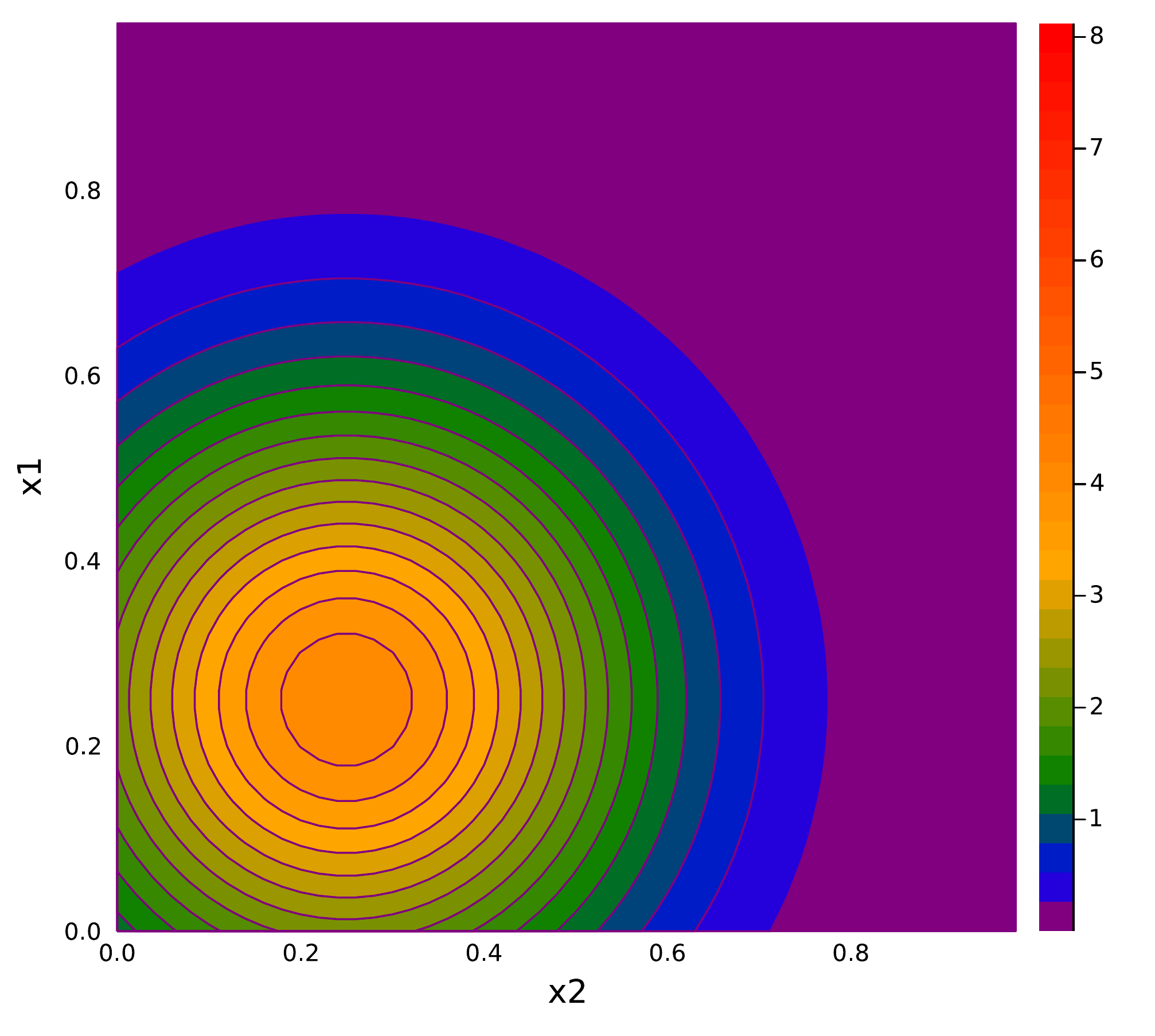} &
\includegraphics[width=.4\textwidth]{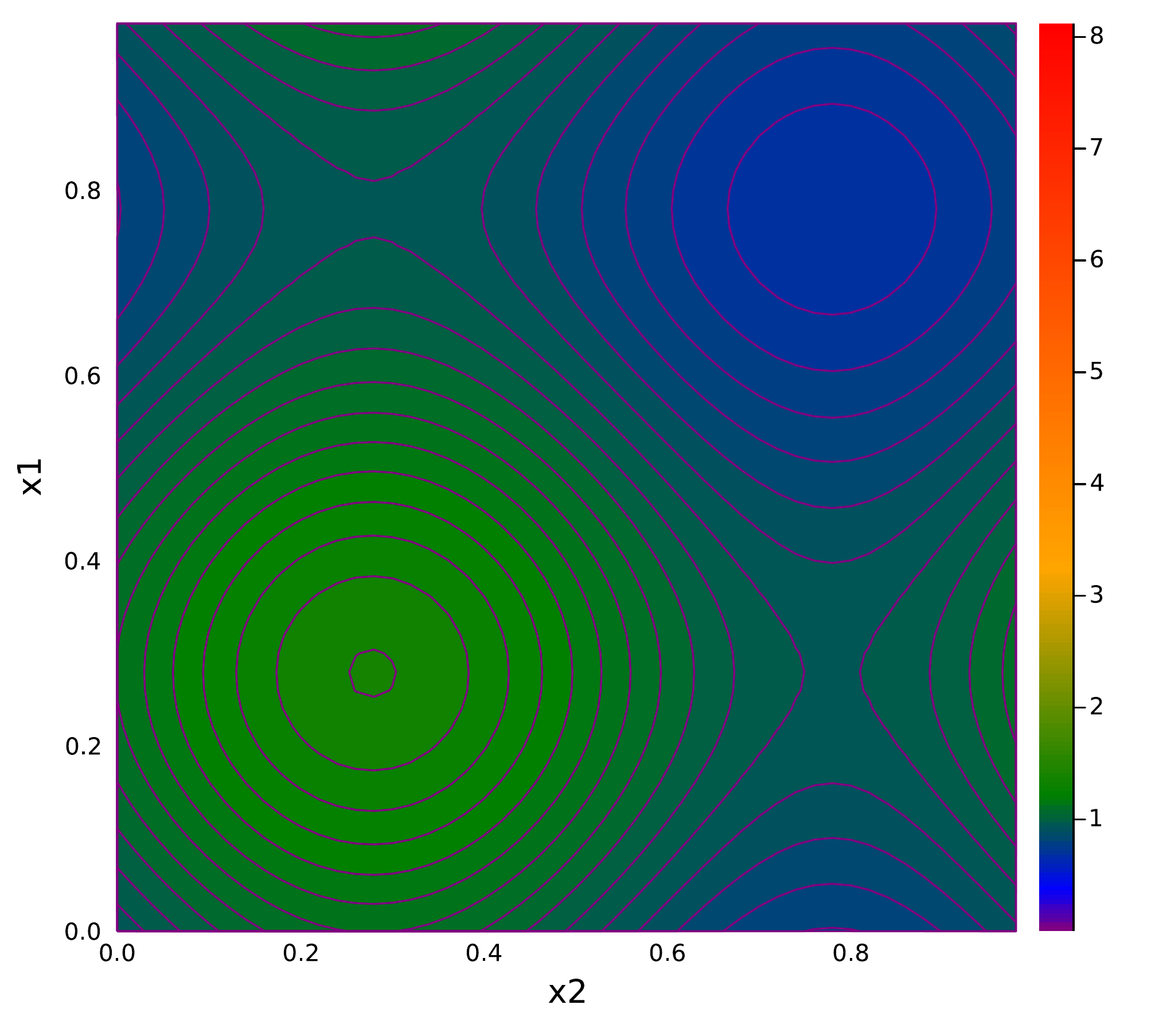} \\
$t=0$ & $t=0.16$ \\\\
\includegraphics[width=.4\textwidth]{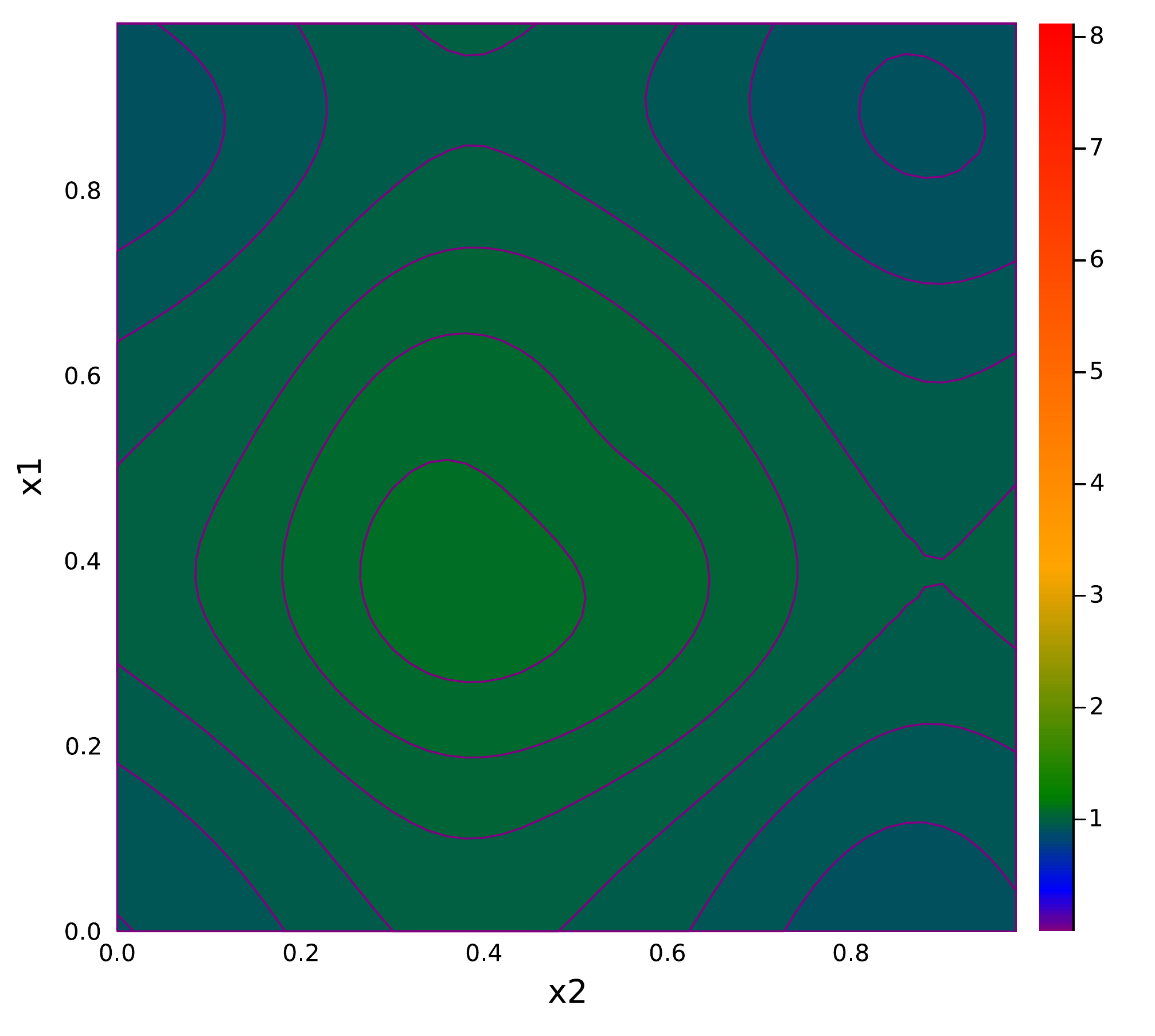} &
\includegraphics[width=.4\textwidth]{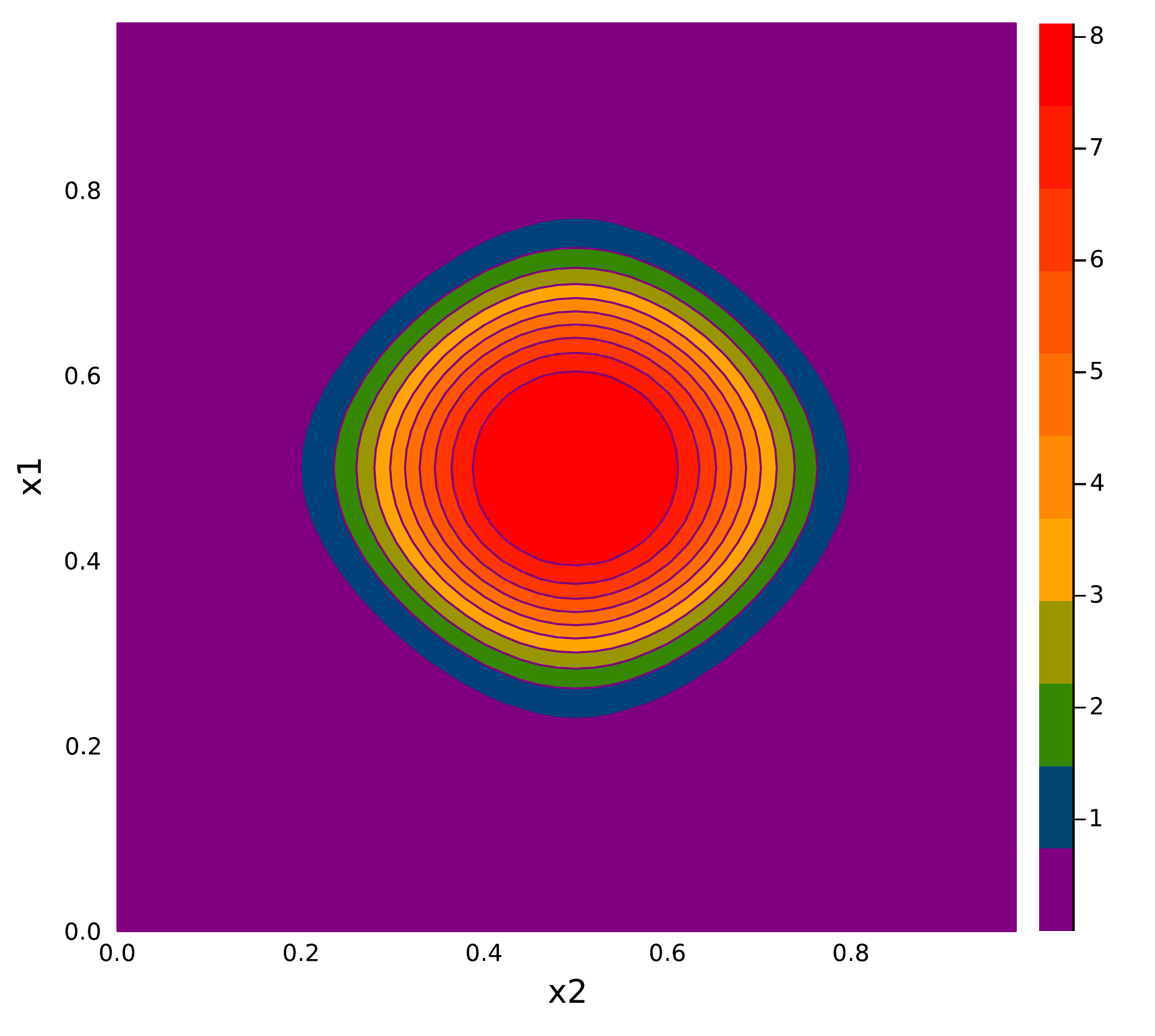} \\
$t=0.33$ & $t=0.5$
\end{tabular}
\end{center}
\caption{Example 3. Solution obtained by \textbf{(PI1)} for the 2d MFG system~\eqref{example3}: contours of density at several time steps.}\label{Fig6}
\end{figure}

\clearpage


\begin{thebibliography}{99}
	
\bibitem{ad}
	Y. Achdou,  I. Capuzzo-Dolcetta. 
	\newblock Mean field games: numerical methods.
	\newblock SIAM J. Numer. Anal. 48 (2010), no. 3, 1136--1162.

\bibitem{acd}
	Y. Achdou, F. Camilli, I. Capuzzo-Dolcetta. 
	\newblock Mean field games: convergence of a finite difference method.
	\newblock  SIAM J. Numer. Anal. 51 (2013), no. 5, 2585-2612.

\bibitem{achdouCetraro}
	Y. Achdou, P. Cardaliaguet, F. Delarue, A. Porretta, F. Santambrogio.
  	\newblock Mean Field Games: Cetraro, Italy 2019.
  	\newblock Springer Nature (2020), volume 2281.

\bibitem{achdou2014partialmacroecon}
	Y. Achdou, F. Buera, J.~M. Lasry, P.-L. Lions, B. Moll.
	\newblock Partial differential equation models in macroeconomics.
  	\newblock Philos. trans., Math. phys. eng. sci. volume 372, 2028 (2014).

\bibitem{achdoulasrycrowd}
	Y. Achdou,  J.~M. Lasry. 
	\newblock Mean field games for modeling crowd motion.
	\newblock In Contributions to partial differential equations and applications, Springer, Cham (2019), 17--42.

\bibitem{achdou2020}
Y.~Achdou, M.~Lauri{\`e}re.
\newblock Mean field games and applications: Numerical aspects.
\newblock {Mean Field Games}, 249--307, 2020.

\bibitem{achdou2016}
Y.~Achdou, M.~Lauri{\`e}re.
\newblock Mean field type control with congestion.
\newblock { Appl. Math. Optim.} 73 (2016), no. 3, 393--418.

\bibitem{achdou2015}
Y.~Achdou, M.~Lauri{\`e}re.
\newblock On the system of partial differential equations arising in mean field type control.
\newblock { Discret. Contin. Dyn. Syst.} 35 (2015), no. 9, 38--79.

\bibitem{all}
Y.~Achdou, M.~Lauri{\`e}re, P.~L. Lions.
\newblock Optimal control of conditioned processes with feedback controls. \newblock { J. Math. Pures Appl.} 148 (2021), 308--341.

\bibitem{achdouporretta2018}
Y.~Achdou, A.~Porretta.
\newblock Mean field games with congestion.
\newblock Ann. Inst. Henri Poincare (C) Anal. Non Lineaire. 35 (2018), 443--480. 

\bibitem{alla}
A. Alla, M. Falcone, D. Kalise. An efficient policy iteration algorithm for dynamic programming equations. \newblock{ SIAM J. Sci. Comput.} 37 (2015), no. 1, A181--A200. 

\bibitem{almulla2017}
	N. Almulla, R. Ferreira, D. Gomes. 
	\newblock Two numerical approaches to stationary mean-field games.
	\newblock {Dyn. Games Appl.} 7.4 (2017), 657--682.


\bibitem{ambrose2018}
D.~M. Ambrose.
\newblock Strong solutions for time-dependent mean field games with
  non-separable Hamiltonians.
\newblock{J. Math. Pures Appl.}
  113 (2018), 141--154.

\bibitem{ambrose2020}
D.~M. Ambrose. Existence theory for non-separable mean field games in Sobolev
  spaces.
\newblock {Indiana U. Math. J}, to appear.

\bibitem{ambrose2021}
D.~M. Ambrose, A.~R. M{\'e}sz{\'a}ros. Well-posedness of mean field games master equations involving
  non-separable local Hamiltonians.
\newblock {arXiv:2105.03926}, 2021.


\bibitem{andreev}
	R. Andreev.
	\newblock Preconditioning the augmented Lagrangian method for in stationary mean field games with diffusion
	\newblock {SIAM J. ScI. Comput.} 39 (2017), no. 6, A2763--A2783.
	
\bibitem{b}
R. Bellman. {Dynamic Programming}. Princeton Univ. Press, Princeton, 1957.


\bibitem{benamoucarlier}
	J.~D. Benamou, G. Carlier.
	\newblock Augmented Lagrangian methods for transport optimization, mean field games and degenerate elliptic equations
	\newblock {J. Optim. Theory Appl.} 167 (2015), no. 1, 1--26.
	
	
\bibitem{ben}
A. Bensoussan, J. Frehse, P. Yam.
{Mean field games and mean field type control theory}. Springer Briefs in Mathematics, New York, 2013. 

\bibitem{bmz}
O. Bokanowski, S. Maroso, H. Zidani. Some convergence results for Howard's algorithm. \newblock {SIAM J. Numer. Anal.} 47 (2009), no. 4, 3001--3026. 

\bibitem{bhp}
J.~F. Bonnans, S.~Hadikhanloo, L.~Pfeiffer.
\newblock Schauder estimates for a class of potential mean field games of
  controls.
\newblock { Appl. Math. Optim.} 83 (2021), no. 3, 1431--1464.


\bibitem{bricenostatio}
	L. M. Brice{\~ n}o-Arias, D. Kalise, F. J. Silva. 
	\newblock Proximal methods for stationary mean field games with local couplings.
	\newblock  SIAM J. Control Optim. 56 (2018), no. 2, 801--836.
	

\bibitem{bricenodyn} 
	L. M. Brice{\~ n}o-Arias, D. Kalise, Z. Kobeissi, M. Lauri{\`e}re, A. M. Gonz{\'a}lez, F. J. Silva. 
	\newblock On the implementation of a primal-dual algorithm for second order time-dependent mean field games with local couplings.
	\newblock ESAIM Proc. Surveys, 65 (2019), 330--348.



\bibitem{ccg}
S. Cacace, F. Camilli, A. Goffi. A policy iteration method for Mean Field Games,  ESAIM Control Optim. Calc. Var. 27 (2021) 85.

\bibitem{ct}
F.~Camilli, Q.~Tang.
\newblock Rates of convergence for the policy iteration method for mean field
  games systems.
\newblock {J. Math. Anal. Appl.} 512 (2022), no. 1, 126--138.



\bibitem{ch}
P.~Cardaliaguet, S. Hadikhanloo. Learning in mean field games: the fictitious play. ESAIM Control Optim. Calc. Var. 23 (2017),  no. 2, 569--591.

\bibitem{ccp}
P.~Cardaliaguet, M.~Cirant,  A.~Porretta.
 Splitting methods and short time existence for the master equations
  in mean field games. arXiv:2001.10406.


\bibitem{carlinisilva1}
	E. Carlini, F.~J. Silva. 
	\newblock A fully discrete semi-Lagrangian scheme for a first order mean field game problem.
	\newblock SIAM J. Numer. Anal. 52 (2014), no. 1, 45--67.

\bibitem{carlinisilva2}
	E. Carlini, F.~J. Silva.
	\newblock A semi-Lagrangian scheme for a degenerate second order mean field game system. 
	\newblock Discrete Contin. Dyn. Syst. 35 (2015),  no. 9, 4269--4292.

\bibitem{carmona2018}
R.~Carmona, F.~Delarue.
Probabilistic theory of mean field games with applications. I, volume
83 of Probability theory and Stochastic modelling, 2018.


%	\bibitem{cc}
%	Cacace, S.;  Camilli, F. A generalized Newton method for homogenization
%	of Hamilton-Jacobi equations, SIAM J. Sci. Comput. 38 (2016), no. 6,
%	A3589-A3617.
%	
%	\bibitem{cs}
%	Carlini, E.; Silva, F. A semi-Lagrangian scheme for a degenerate second order mean
%	field game system. Discrete Contin. Dyn. Syst. 35 (2015), no. 9, 4269-4292.
%	
%	\bibitem{cdl}
%	Carmona, R.; Delarue, F.; Lacker, D. Mean field games of timing and models for bank runs. Appl. Math. Optim. 76 (2017), no. 1, 217---260.
%	\bibitem{Cirant:ab}

\bibitem{cirant2020}
M.~Cirant, R.~Gianni, P.~Mannucci.
\newblock Short-time existence for a general backward--forward parabolic system arising from mean-field games.
\newblock {Dyn. Games Appl.} 10 (2020), no. 1, 100--119.

\bibitem{crr}
C.~Cuchiero, C.~Reisinger, S.~Rigger.
\newblock Optimal bailout strategies resulting from the drift controlled
  supercooled Stefan problem. arXiv:2111.01783, 2021.



\bibitem{ferreira2021existence}
R.~Ferreira, D.~Gomes, T.~Tada.
\newblock Existence of weak solutions to time-dependent mean-field games.
\newblock {Nonlinear Anal.} 212 (2021), 112470.

\bibitem{fl}
W.~H. Fleming. Some Markovian optimization problems. J. Math. Mech. 12 (1963), 131--140.	

\bibitem{gangbo2021}
W.~Gangbo, A.~R. M{\'e}sz{\'a}ros, C.~Mou, J.~Zhang.
\newblock Mean field games master equations with non-separable Hamiltonians and
  displacement monotonicity.
\newblock {arXiv:2101.12362}, 2021.


\bibitem{Gomes2021}
D.~A. Gomes, J~.Saude.
\newblock Numerical methods for finite-state mean-field games satisfying a monotonicity condition.
\newblock {Appl. Math. Optim.} 83 (2021), no. 1, 51--82.




\bibitem{gomes2015}
D.~A. Gomes, V.~K. Voskanyan.
\newblock Short-time existence of solutions for mean-field games with congestion.
\newblock {J. Lond. Math.} 92 (2015), no. 3, 778--799.




\bibitem{graber2015}
P.~J. Graber.
\newblock Weak solutions for mean field games with congestion.
\newblock {arXiv:1503.04733}, 2015.

\bibitem{gianni1995}
R.~Gianni.
\newblock Global existence of a classical solution for a large class of free
  boundary problems in one space dimension.
\newblock {NoDEA Nonlinear Differ. Equ. Appl.}
  2 (1995), no. 3, 291--321.



\bibitem{hadikhanlooNonatomic}
	S. Hadikhanloo.
	\newblock Learning in anonymous nonatomic games with applications to first-order mean field games. 
	\newblock arXiv:1704.00378, 2017.





	\bibitem{h}
	R. Howard. {Dynamic Programming and Markov Processes}. MIT Press, Cambridge, 1960.
	 
	\bibitem{hcm}
	M. Huang; P.~E. Caines, R.~P. Malhame.  Large-population cost-coupled LQG problems with non uniform agents: Individual-mass behaviour and decentralized $\epsilon$-Nash equilibria. IEEE Trans. Autom. Control. 52 (2007), 1560-1571.
	
	\bibitem{kss}
	B. Kerimkulov, D.~\v{S}i\v{s}ka, L. Szpruch.
	Exponential convergence and stability of Howards's policy improvement algorithm for controlled diffusions, \newblock {SIAM J. Control Optim.} 53 (2020), 1314--1340.
	

	
	\bibitem{LSU}
	 O.~A. Ladyzenskaja, V.~A. Solonnikov, N.~N. Ural'ceva.
     { Linear and quasilinear equations of parabolic type}.
	 Translated from the Russian by S. Smith. Translations of Mathematical
	 Monographs, Vol. 23. American Mathematical Society, Providence, R.I., 1968.
	 
	\bibitem{lauriere2021}
M.~Lauri{\`e}re.
\newblock Numerical methods for mean field games and mean field type control.
\newblock To appear in AMS Proceedings of Symposia in Applied Mathematics, 2021.
	 
	 \bibitem{ll}
	J.~M. Lasry,  P.~L. Lions. Mean field games. \newblock{Jpn. J. Math.} 2 (2007), 229--260.


\bibitem{nurbekyansaude}
	L. Nurbekyan, J. Sa{\'u}de. 
	\newblock Fourier approximation methods for first-order nonlocal mean-field games.
	\newblock Port. Math. 75 (2019), no. 3, 367-396.
	
\bibitem{perolatOMD}
	J. P{\'e}rolat, S. Perrin, R. Elie, M. Lauri{\`e}re, G. Piliouras, M. Geist, K. Tuyls, O. Pietquin.
	\newblock Scaling up Mean Field Games with Online Mirror Descent.
	\newblock Proc. of the 21st International Conference on Autonomous Agents and Multiagent Systems (AAMAS 2022).


\bibitem{perrinFP}
S. Perrin, J. P{\'e}rolat, M. Lauri{\`e}re, M. Geist, R. Elie, O. Pietquin.
\newblock Fictitious play for mean field games: Continuous time analysis and applications.
\newblock Proc. of Advances in Neural Information Processing Systems 33 (NeurIPS 2020).
	


	\bibitem{pu1}
	M.~L. Puterman. On the convergence of policy iteration for controlled diffusions. \newblock {J. Optim. Theory Appl.} 33 (1981), no. 1, 137-144.
	
	
%	\bibitem{pu2}
%	Puterman, M.~L.  Optimal control of diffusion processes with reflection. J. Optim. Theory Appl. 22 (1977), no. 1, 103-116.
	
\bibitem{pb}
	M.~L. Puterman, S.L. Brumelle.
	On the convergence of policy iteration in stationary dynamic programming. 
	Math. Oper. Res. 4 (1979), 60--69.

\bibitem{salhab2017collectivechoice}
R. Salhab, R. P. Malhamé,  J. Le Ny. A dynamic game model of collective choice in multiagent systems.  IEEE Trans. Autom. Control, 63 (2017), no. 3, 768--782.

\bibitem{santos}
M.~S. Santos, J. Rust. Convergence properties of policy iteration. 
\newblock {SIAM J. Control Optim.} 42 (2004), no. 6, 2094--2115.	


\end{thebibliography}
\end{document}